\documentclass{article}

\pdfoutput=1

\usepackage{amsmath,amsthm,amssymb,mathtools}
\usepackage[margin=3cm]{geometry}
\usepackage{enumerate}
\usepackage{charter}
\usepackage{dsfont}
\usepackage[center,figbotcap]{subfigure}
\usepackage[backgroundcolor=black!20]{todonotes}
\usepackage[vlined,ruled,commentsnumbered,linesnumbered,norelsize]{algorithm2e}
\usepackage{booktabs}
\usepackage[binary-units=true]{siunitx}
\usepackage{cases}
\usepackage{nicefrac}
\usepackage{authblk}

\usepackage{tikz}
\usetikzlibrary{shapes}

\usepackage[breaklinks,colorlinks,citecolor=blue,linkcolor=blue,urlcolor=blue,bookmarks=false,pdftitle={Computational
  Aspects of Relaxation Complexity: Possibilities and
  Limitations},pdfauthor={Gennadiy Averkov, Christopher Hojny, and Matthias
  Schymura}]{hyperref}
\usepackage[nameinlink]{cleveref}

\theoremstyle{definition}
\newtheorem{definition}{Definition}
\newtheorem{question}[definition]{Question}
\newtheorem{remark}[definition]{Remark}

\newtheorem{observation}[definition]{Observation}

\theoremstyle{plain}

\newtheorem{theorem}[definition]{Theorem}
\newtheorem{proposition}[definition]{Proposition}
\newtheorem{lemma}[definition]{Lemma}

\newtheorem{conjecture}[definition]{Conjecture}

\DeclareMathOperator{\conv}{conv}

\DeclareMathOperator{\rc}{rc}
\DeclareMathOperator{\rcl}{rc_{\scriptscriptstyle\square}}
\DeclareMathOperator{\rcu}{rc_0}
\DeclareMathOperator{\obs}{Obs}
\DeclareMathOperator{\bd}{bd}
\DeclareMathOperator{\aff}{aff}
\DeclareMathOperator{\vol}{vol}

\newcommand{\R}{\mathds{R}}
\newcommand{\Z}{\mathds{Z}}

\newcommand{\Q}{\mathds{Q}}
\newcommand{\cO}{\mathcal{O}}

\newcommand{\I}{\mathcal{I}}

\DeclarePairedDelimiter{\card}{\lvert}{\rvert}
\newcommand{\one}{\mathbf{1}}
\newcommand{\zero}{\mathbf{0}}
\newcommand{\st}{:}		
\newcommand{\T}{^\intercal}
\newcommand{\sprod}[2]{{#1}\T{#2}}
\newcommand{\define}{\coloneqq}

\newcommand{\interior}[1]{\mathrm{int}(#1)}

\newcommand{\ceiling}[1]{\left\lceil #1\right\rceil}
\newcommand{\setcond}[2]{\left\{ #1 \,:\, #2 \right\}}

\begin{document}

\title{Computational Aspects of Relaxation Complexity: Possibilities and Limitations}
\author[1]{Gennadiy Averkov}
\author[2]{Christopher Hojny}
\author[1]{Matthias Schymura}

\affil[1]{
  BTU Cottbus-Senftenberg\\
  Platz der Deutschen Einheit 1\\
  03046 Cottbus, Germany\\
  \newline
  \textit{email:} \{averkov, schymura\}@b-tu.de
}
\affil[2]{
  Eindhoven University of Technology\\
  Combinatorial Optimization Group\\
  PO Box~513\\
  5600 MB Eindhoven, The~Netherlands\\
  \textit{email:} c.hojny@tue.nl
}

\renewcommand\Affilfont{\itshape\small}

\maketitle

\begin{abstract}
  The relaxation complexity~$\rc(X)$ of the set of integer points~$X$
  contained in a polyhedron is the smallest number of facets of any
  polyhedron~$P$ such that the integer points in~$P$ coincide with~$X$.
  It is a useful tool to investigate the existence of compact linear
  descriptions of~$X$.
  In this article, we derive tight and computable upper bounds on~$\rc_\Q(X)$, a variant
  of~$\rc(X)$ in which the polyhedra~$P$ are required to be rational,
  and we show that~$\rc(X)$ can be computed in polynomial time if~$X$ is
  2-dimensional.
  Further, we investigate computable lower bounds on~$\rc(X)$ with the
  particular focus on the existence of a finite set~$Y \subseteq \Z^d$ such
  that separating~$X$ and~$Y \setminus X$ allows us to deduce $\rc(X) \geq k$.
  In particular, we show for some choices of~$X$ that no such finite set~$Y$
  exists to certify the value of~$\rc(X)$, providing a negative answer
  to a question by Weltge (2015).
  We also obtain an explicit formula for~$\rc(X)$ for specific classes of sets~$X$ and present the first practically applicable approach to compute~$\rc(X)$ for sets~$X$ that admit a finite certificate.

\end{abstract}

\section{Introduction}

A successful approach for solving discrete optimization problems is
based on integer programming techniques.
To this end, (\romannumeral1) a suitable encoding~$X \subseteq \Z^d$ of the
discrete problem's solutions together with an objective~$c \in \R^d$
has to be selected, and~(\romannumeral2) a linear system~$Ax \leq b$, $Cx = f$ defining a
polyhedron~$P \subseteq \R^d$ has to be found such that~$P \cap \Z^d = X$.
In the following, we refer to such a polyhedron~$P$ as a \emph{relaxation}
of~$X$.
Then, the discrete problem can be tackled by solving the integer program
$\max \{ \sprod{c}{x} : Ax \leq b,\; Cx = f,\; x \in
\Z^d \}$, which can be solved, e.g., by branch-and-bound or branch-and-cut
techniques (see, e.g., Conforti et al.~\cite{ConfortiEtAl2014} or
Schrijver~\cite{schrijver1986theory}).

To solve these integer programs efficiently, the focus in Step~(\romannumeral2) was mostly on identifying facet defining inequalities of the integer hull of~$P$.
Such inequality systems, however, are typically exponentially large and one
may wonder about the minimum number of facets of any relaxation of~$X$,
which allows to compare different encodings~$X$ of the discrete problem.
Kaibel \& Weltge~\cite{kaibelweltge2015lowerbounds} called this quantity
the \emph{relaxation  complexity} of~$X$, denoted~$\rc(X)$, and showed that
certain encodings of, e.g., the traveling salesman problem or
connected subgraphs, have exponentially large relaxation complexity.
They also introduced the quantity~$\rc_{\Q}(X)$, which is the smallest
number of facets of a \emph{rational} relaxation of~$X$, and posed the
question whether $\rc(X) = \rc_\Q(X)$ holds in general.
Recently, this was answered affirmatively for~$d \leq 4$,
see~\cite{averkovschymura2020complexity}.
The same authors also showed that~$\rc(X)$ is computable if~$d \leq 3$
using an algorithm with potentially super-exponential  worst case running time.
Computability for~$d \geq 4$ and explicit formulas for~$\rc(X)$ for
specific sets~$X$, however, are still open problems.

In this article, we follow the line of research started in~\cite{averkovschymura2020complexity} and derive a tight upper
bound on~$\rc_\Q(X)$ for arbitrary dimensions, which is
based on a robustification of~$\rc_\Q(X)$ against numerical errors
(Section~\ref{sec:ub}).
We also point out when this upper bound can be used to compute~$\rc_\Q(X)$ exactly.
Computable lower bounds on $\rc(X)$ are dealt with in Section~\ref{sec:rc-square}, with a particular focus on the question about the existence of finite certificates for the size of a minimal relaxation.
In Section~\ref{sec:AlgorithmDimension2} we focus on dimension~$d=2$ and show
that there is a (weakly) polynomial time algorithm to compute~$\rc(X)$ in this case.
Section~\ref{sec:box} derives an explicit formula for the relaxation complexity
for integer points in rectangular boxes and crosspolytopes.
In Section~\ref{sec:MIP} we conclude the article with a discussion of numerical experiments.
We also discuss our first practically applicable implementation to compute $\rc(X)$ in cases where this quantity is known to have a finite certificate.

An extended abstract of this article appeared as~\cite{averkovhojnyschymura021ipco} in the conference proceedings of IPCO 2021.

\medskip
\noindent\emph{Notation and Terminology.}
For a set~$X \subseteq \R^d$, we denote by~$\conv(X)$ its convex hull and
by~$\aff(X)$ its affine hull.
The boundary and interior of~$X$ are denoted by~$\bd(X)$
and~$\interior{X}$, respectively.
A set~$X \subseteq \Z^d$ is called \emph{lattice-convex} if \mbox{$\conv(X) \cap \Z^d = X$}.
The dimension of a set~$X$, denoted by $\dim(X)$, is the dimension of its affine hull.
For two sets~$X, Y \subseteq \R^d$, $X + Y = \{x + y : x \in X,\; y \in Y\}$ is their Minkowski sum.
We define $[n] \define \{1, \dots, n\}$, $[n]_0 \define
\{0,1,\dots,n\}$, and~$e_i$ to be the~$i$-th standard unit vector in~$\R^d$.

A crucial concept is that of observers~\cite{averkovschymura2020complexity}.
For~$X \subseteq \Z^d$ with~$\conv(X) \cap \Z^d = X$,
an \emph{observer} is a point~$z \in \Z^d \setminus X$ such that $\conv(X \cup \{z\}) \cap \Z^d = X \cup \{z\}$.
The set of all observers of~$X$ is denoted by~$\obs(X)$.
The relevance of the observers for finding~$\rc(X)$ comes from the fact that any linear system~$Ax \leq b$ separating~$X$
and~$\obs(X)$ also separates~$X$ and~$\Z^d \setminus X$,
see~\cite{averkovschymura2020complexity}.
This motivates, for~$X, Y \subseteq \Z^d$, to
introduce~$\rc(X,Y)$ (resp.~$\rc_\Q(X,Y)$) as the smallest number of inequalities in a (rational) system~$Ax \leq b$ separating~$X$ and~$Y \setminus X$.

\medskip
\noindent\emph{Related Literature.}
Kaibel \& Weltge~\cite{kaibelweltge2015lowerbounds} introduced the notion
of relaxation complexity and derived a concept for deriving lower bounds
on~$\rc(X)$.
Computability of~$\rc(X)$, for~$d = 2$, has been shown by
Weltge~\cite{weltge2015diss} who also derived a lower bound on~$\rc(X)$
only depending on the dimension of~$X$.
In~\cite{averkovschymura2020complexity}, this lower bound has been improved
and computability also for~$d=3$ has been established.
The interplay between the number of inequalities in a relaxation and the
size of their coefficients has been investigated in~\cite{Hojny2020}; see
also~\cite{hojny2018strong} for a lower bound on the relative size of
coefficients in a relaxation.
For~$X \subseteq \{0,1\}^d$, Jeroslow~\cite{jeroslow1975ondefining} derived
an upper bound on~$\rc(X, \{0,1\}^d)$, which is an important subject in the
area of social choice, see, e.g., Hammer et
al.~\cite{hammeribarakipeled1981threshold} and~Taylor \&
Zwicker~\cite{taylorzwicker1999simplegames}.
The problem of separating two sets by a function from a given family is
known as the binary classification problem in the machine learning community.
The support vector machine is the binary classification by means  of linear
functions, see, e.g., Cristiani \& Shawe-Taylor~\cite{cristianini_shawe-taylor_2000}.
Our investigations can be interpreted as a variation on the
support-vector-machine theme, as we deal with the binary classification
using the maximum of (a possibly small) number of linear functions.

\section{Computable Bounds on the Relaxation Complexity}
\label{sec:ub}

Let~$X \subseteq \Z^d$ be a lattice-convex set.
Throughout the article we assume that $X$ is finite and full-dimensional.
Finiteness is a crucial property for many arguments to work, but at the same time it is natural since the usual combinatorial optimization problem deals with integer points in a certain polytope.
Assuming $X$ to be full-dimensional is no loss of generality: The relaxation complexity as defined by Kaibel \& Weltge~\cite{kaibelweltge2015lowerbounds} only concerns the inequalities of a linear description of~$X$, while equations can be added at no cost.

Besides the theoretical lower bounds on~$\rc(X)$ based on hiding sets (see Section~\ref{sect:hiding-sets} below), to
the best of our knowledge no systematic way for deriving \emph{algorithmically
computable} lower and upper bounds on~$\rc(X)$ and~$\rc_\Q(X)$ has been
discussed so far.
One of our goals in this work is to make progress on this problem.
To this end, note that~$\rc(X,Y) \leq \rc(X)$ for any~$Y \subseteq \Z^d$.
This bound is best possible if~$\rc(X,Y) = \rc(X)$, which leads us to the
following concept.

\begin{definition}
  Let~$X \subseteq \Z^d$ be lattice-convex.
  Then, $Y \subseteq \Z^d$ is called a \emph{certificate} for~$\rc(X)$,
  if~$\rc(X, Y) = \rc(X)$.
\end{definition}

Analogously, we define certificates for~$\rc_\Q(X)$ and the related operators that we describe below.
In view of~\cite[Prop.~4.3]{averkovschymura2020complexity}, a certificate for~$X$ is given
by its set of observers~$\obs(X)$.
Since~$\obs(X)$ can be infinite, we might not be able to systematically
compute~$\rc(X, \obs(X))$.
If~$Y$ is finite, however, $\rc(X,Y)$ can be computed using a mixed-integer
program, see~\cite{averkovschymura2020complexity} and
Section~\ref{sec:MIP}, which suggests the following lower bound on~$\rc(X)$:

For~$t \in \Z_{>0}$, let $B_t = [-t,t]^d \cap \Z^d$ be the set of integer points
with coordinates bounded by~$t$ in absolute value.
Now, it is clear that $\rc(X,B_t) \leq \rc(X,B_{t'})$, for every $t \leq t'$.
Thus, the parameter
\[
\rcl(X) \define \max_{t \in \Z_{>0}} \,\rc(X,B_t)
\]
is well-defined and satisfies~$\rcl(X) \leq \rc(X)$.
Moreover, this is an identity if and only if~$\rc(X)$ admits a \emph{finite} certificate.

In the remainder of this section, we also introduce a computable (and
eventually tight) upper bound on~$\rc_\Q(X)$ and we show that this always admits a finite certificate.
Moreover, we also review the idea of so-called hiding sets and introduce hiding graphs which will be most useful for our computational approach discussed in Section~\ref{sec:MIP}.

\subsection{A computable upper bound on \texorpdfstring{$\rc_\Q(X)$}{rc\_Q(X)}}

To find an upper bound on~$\rc_\Q(X)$,
we introduce a robustification of~$\rc(X)$ by enforcing that a
relaxation is not allowed to support~$\conv(X)$.
To make this precise, for~$\varepsilon > 0$, let~\mbox{$X_\varepsilon \define X + \varepsilon\lozenge_d$}, where~$\varepsilon\lozenge_d = \left\{0, \pm
  \varepsilon e_1,\ldots,\pm \varepsilon e_d\right\}$ is the discrete
$\ell_1$-ball with radius~$\varepsilon$.
By construction, we clearly have $X \subseteq \interior{\conv(X_\varepsilon)}$.

\begin{definition}
  Let~$X \subseteq \Z^d$ be lattice-convex and let~$\varepsilon > 0$.
  A polyhedron~$Q \subseteq \R^d$ is called an
  \emph{$\varepsilon$-relaxation} of~$X$ if $X_\varepsilon \subseteq Q$
  and $\interior{Q} \cap \Z^d = X$.
  The \emph{$\varepsilon$-relaxation complexity} $\rc_{\varepsilon}(X)$
  is the smallest number of facets of an~$\varepsilon$-relaxation of~$X$.
\end{definition}

Alternatively speaking, an $\varepsilon$-relaxation of~$X$ is a polyhedron~$Q$ such that $\interior{Q} \cap \Z^d = X$, and for which the minimum of $\|x-z\|_1$, over any $x \notin \interior{Q}$ and $z \in X$, is at least~$\varepsilon$.
Further, in complete analogy to~$\rc(X,Y)$, we define~$\rc_{\varepsilon}(X,Y)$ to be the smallest number of inequalities necessary to separate~$X_{\varepsilon}$ and~$Y \setminus
X_{\varepsilon}$.

Our first observation is that~$\rc_\varepsilon(X)$ always admits a \emph{finite} certificate~$Y^\varepsilon \subseteq \Z^d$.
The existence of such a finite certificate is noteworthy,
because, in general, $\rc(X)$ might not admit a finite certificate, see Theorem~\ref{thm:rc-square-simplex} below.

\begin{lemma}
  \label{lem:epsFiniteTime}
  Let~$X \subseteq \Z^d$ be lattice-convex and let $\varepsilon >0$ be such that~$X$ admits an~$\varepsilon$-relaxation.
  Then, there is an explicitly computable finite set $Y^\varepsilon \subseteq \Z^d$ such that $\rc_\varepsilon(X) = \rc_\varepsilon(X,Y^\varepsilon)$.
\end{lemma}

\begin{proof}
We show that there exists an explicit constant $c_{d,\varepsilon,X} > 0$ such that every $\varepsilon$-relaxation~$Q$ of $X$ is contained in $P_{X,\varepsilon} \define \conv(X) + c_{d,\varepsilon,X} \cdot [-1,1]^d$.
Then, the finite set $Y^\varepsilon \define P_{X,\varepsilon} \cap \Z^d$ is explicitly computable such that $\rc_\varepsilon(X) = \rc_\varepsilon(X,Y^\varepsilon)$.

In order to find $P_{X,\varepsilon}$, we first note that by definition $X + \varepsilon \lozenge_d \subseteq \conv(X_\varepsilon)$.
Now, let $Q$ be an arbitrary $\varepsilon$-relaxation of~$X$ and let $q \in Q$ be some point therein.
As~$Q$ weakly separates~$X$ from~$\Z^d \setminus X$, the set $\conv(X \cup \{q\})$ contains at most $\card{X}+1$ integer points.
For~$v \in X$, let
\[
K_{v,q} \define \conv\left((v + \varepsilon \lozenge_d) \cup \{q,2v-q\}\right) \ \textrm{and} \ K'_{v,q} \define \conv\left((v + \varepsilon \lozenge_d) \cup \{q\}\right).
\]
The set $K_{v,q}$ is convex and centrally symmetric around the integer point~$v$.
Moreover, $K'_{v,q} \subseteq Q$, so that we would get a contradiction if $\card{K'_{v,q} \cap \Z^d} > \card{X} + 1$.
A classical extension of Minkowski's first fundamental theorem in the Geometry of Numbers is due to van der Corput~\cite{vdcorput1936verallg} who proved that, for every convex body $K \subseteq \R^d$ that is centrally symmetric around an integer point, the inequality $\vol(K) \leq 2^d \cdot \card{K \cap \Z^d}$ holds.
Without loss of generality, assume that the $\ell_\infty$-norm $N \define \|v-q\|_\infty$ of~$v-q$ is attained at the last coordinate.
Then, $K_{v,q}$ contains the crosspolytope
\[
C_{v,q} \define \conv(\{v \pm \varepsilon e_1,\ldots,v \pm \varepsilon
e_{d-1},q,2v-q\}),
\]
and so by van der Corput's result we have
\begin{align*}
\card{K'_{v,q} \cap \Z^d} &\geq \tfrac{1}{2} \card{K_{v,q} \cap \Z^d} \geq \tfrac{1}{2^{d+1}} \vol(K_{v,q}) \geq \tfrac{1}{2^{d+1}} \vol(C_{v,q}) = \tfrac{1}{2^{d+1}} \tfrac{2^d}{d!} \varepsilon^{d-1} N.
\end{align*}
To avoid the discussed contradiction, the last term must be upper bounded by $\card{X}+1$, which translates to the inequality
\[
\|v-q\|_\infty  = N \leq \frac{2 \, d! \, (\card{X}+1)}{\varepsilon^{d-1}} \eqqcolon c_{d,\varepsilon,X}.
\]
Since $q \in Q$ and $v \in X$ were arbitrary, this indeed shows that the $\varepsilon$-relaxation~$Q$ of~$X$ is contained in $\conv(X) + c_{d,\varepsilon,X} \cdot [-1,1]^d = P_{X,\varepsilon}$, as desired.
\end{proof}

Note that $\rc_\varepsilon(X) \geq \rc_{\varepsilon'}(X)$, for every $\varepsilon \geq \varepsilon' > 0$.
Hence, the parameter
\[
  \rcu(X) \define \min_{\varepsilon > 0} \rc_\varepsilon(X)
\]
is well-defined.
The main advantage of~$\rcu(X)$ and the previously defined~$\rcl(X)$ is that they are limits of explicitly computable numbers which sandwich~$ \rc(X)$:

\begin{theorem}
\label{thm:epsrelaxProperties}
Let $X \subseteq \Z^d$ be lattice-convex and full-dimensional.
\begin{enumerate}
 \item\label{epsrelaxProperties2} If $\varepsilon > 0$ is rational, then $\rc_\varepsilon(X)$ can be computed in finite time.
 \item\label{epsrelaxProperties2b} For every $t \in \Z_{>0}$, the number $\rc(X,B_t)$ can be computed in finite time.
 \item\label{epsrelaxProperties3} One has
  \begin{align}
  \rcl(X) \leq \rc(X) \leq \rc_\Q(X) = \rcu(X).\label{eqn:inequalities-bounds}
  \end{align}
\end{enumerate}
\end{theorem}

\begin{proof}
Lemma~\ref{lem:epsFiniteTime} provides us with a finite and explicitly computable set $Y^\varepsilon \subseteq \Z^d$ such that $\rc_\varepsilon(X) = \rc_\varepsilon(X,Y^\varepsilon)$.
If $\varepsilon \in \Q$, then $X_\varepsilon \subseteq \Q^d$, which means that we ask to minimally separate the finite rational sets $X_\varepsilon$ and $Y^\varepsilon \setminus X_\varepsilon$ from one another.
The arguments in~\cite[Prop.~4.9]{averkovschymura2020complexity} show how this can be done via a bounded mixed-integer program and Claim~\ref{epsrelaxProperties2} follows.
As $B_t$ is finite, the same can be done for $\rc(X,B_t)$ implying Claim~\ref{epsrelaxProperties2b}.

For Claim~\ref{epsrelaxProperties3}, observe that for every $t \in \Z_{>0}$ we clearly have $\rc(X,B_t) \leq \rc(X)$, and therefore $\rcl(X) \leq \rc(X) \leq \rc_\Q(X)$.
We thus need to show $\rc_\Q(X) = \rcu(X)$.
For one inequality, let~$Q \subseteq \R^d$ be a rational relaxation of~$X$ having~$\rc_\Q(X)$
facets and facet description~$Ax \leq b$.
Since~$Q$ is a rational polyhedron, $Q$ is necessarily bounded (otherwise
it would contain infinitely many integral points).
Thus, there exists~$\delta > 0$ such that for each~$y \in \Z^d \setminus
X$ there exists an inequality~$\sprod{a}{x} \leq \beta$ in~$Ax \leq b$
with~$\sprod{a}{y} \geq \beta + \delta$.
Consequently, we can increase~$\beta$ slightly to get another
relaxation~$Q'$ with~$\rc_\Q(X)$ facets.
This shows that there exists~$\varepsilon' > 0$ such that~$X_{\varepsilon'} \subseteq Q'$ and $\interior{Q'} \cap \Z^d = X$, i.e., $Q'$ is an~$\varepsilon'$-relaxation of~$X$, and thus $\rcu(X) \leq \rc_{\varepsilon'}(X) \leq \rc_\Q(X)$.

For the reverse inequality, we fix an $\varepsilon > 0$ and note that every $\varepsilon$-relaxation of~$X$ is bounded (see the proof of Lemma~\ref{lem:epsFiniteTime}).
We may thus perturb any such $\varepsilon$-relaxation slightly into a rational relaxation of~$X$ with equally many facets.
As a result, we get $\rc_\Q(X) \leq \rc_\varepsilon(X)$, for every $\varepsilon > 0$.
\end{proof}

The main message of Theorem~\ref{thm:epsrelaxProperties} is that, for every rational~$\varepsilon > 0$, the number $\rc_\varepsilon(X)$ is an explicitly computable upper bound on~$\rc(X)$ and~$\rc_\Q(X)$, and that these upper bounds converge to $\rc_\Q(X)$ with $\varepsilon \to 0$.
However, without further information we do not know whether the computed value $\rc_\varepsilon(X)$ agrees with $\rc_\Q(X)$.
A situation in which we are sure when to stop computing $\rc_\varepsilon(X)$ for decreasing values of $\varepsilon > 0$ is when we can also compute an eventually matching lower bound:

\begin{theorem}
\label{thm:computingrcX-with-condition}
  Let~$X \subseteq \Z^d$ be lattice-convex and full-dimensional.
  If $\rcl(X) = \rc_{\Q}(X)$, then there is a finite algorithm that computes~$\rc(X)$.
\end{theorem}

\begin{proof}
  Let~$(\varepsilon_t)_{t \in \Z_{>0}} \subseteq \Q$ be a rational strictly decreasing null
  sequence.
  Then, by Theorem~\ref{thm:epsrelaxProperties}, Part~\ref{epsrelaxProperties3}, there is~$t' \in \Z_{>0}$ such that~$\rc_\Q(X) =
  \rc_{\varepsilon_{t}}(X)$ for every~$t \geq t'$.
  If~$\rcl(X) = \rc_\Q(X)$, there exists~$t^\star \in \Z_{>0}$ such
  that~$\rc_\Q(X) = \rc(X, B_t)$, for every $t \geq t^\star$.
  Thus, for $t = \max\{t', t^\star\}$, we get $\rc(X, B_t) =
  \rc_{\varepsilon_t}(X)$.
  Since both quantities can be computed in finite time due to
  Theorem~\ref{thm:epsrelaxProperties}, the assertion
  follows.  
\end{proof}

Our constructions suggest infinite iterative procedures that produce sequences of integer values converging to $\rcl(X)$ and $\rc_\Q(X)$ after finitely many steps.
However, the existence of such procedures per se does not resolve the computability of $\rcl(X)$ or $\rc_\Q(X)$.
For computability, one would additionally need to be able to decide when
the integer sequence achieves the value it finitely converges to.

\subsection{Hiding sets and hiding graphs}
\label{sect:hiding-sets}

Let~$X \subseteq \Z^d$ be lattice-convex.
To derive a lower bound on~$\rc(X)$, Kaibel \&
Weltge~\cite{kaibelweltge2015lowerbounds} introduced the notion of hiding
sets.
A set~$H \subseteq (\aff(X) \cap \Z^d) \setminus X$ is called a \emph{hiding set} for~$X$ if, for all distinct~\mbox{$x,y \in H$}, we have~$\conv(\{x,y\})
\cap \conv(X) \neq \emptyset$.
No valid inequality for~$X$ can separate any such~$x$ and~$y$ simultaneously,
showing that
\begin{align}
\rc(X) \geq \card{H}, \quad \textrm{for every hiding set } H \textrm{ for }X.\label{eqn:hiding-set-bound}
\end{align}
Kaibel \& Weltge used this bound to show that the set of integer points in the traveling salesman polytope has exponential relaxation complexity (see~\cite[Thm.~3]{kaibelweltge2015lowerbounds}).
We refer to the best possible lower bound based on hiding sets as the
\emph{hiding set bound} of~$X$,
\[
  H(X) \define \max \{ \card{H} \st H \text{ is a hiding set for } X\}.
\]
It is not immediate how to explicitly compute this number, because~$X$
might admit infinitely many hiding sets.
However, a maximum size hiding set can always be found
among the observers of~$X$.

\begin{lemma}
\label{lem:hiding-sets-in-observers}
Let $X \subseteq \Z^d$ be lattice-convex, let $H$ be a hiding set for~$X$, and assume that there exists an $h \in H \setminus \obs(X)$.
Then, for every $y \in \obs(X) \cap \conv(\{h\} \cup X)$, the set $H \setminus \{h\} \cup \{y\}$ is a hiding set for~$X$ as well.

In particular, there is a maximum size hiding set contained in~$\obs(X)$.
\end{lemma}

\begin{proof}
For a point $z \notin \conv(X)$, let $C_z \define z + \left\{\sum_{x \in X} \lambda_x(x-z) : \lambda_x \geq 0\right\}$ be the smallest convex cone with apex~$z$ containing~$X$.
Further, \mbox{$U_z \define C_z \setminus \conv(\{z\} \cup X)$} contains all points $w \notin \conv(X)$ such that $\conv(\{z,w\}) \cap \conv(X) \neq \emptyset$.
By the choice of~$y$, we have $\conv(\{y\} \cup X) \subseteq \conv(\{h\} \cup X)$, and in particular $U_h \subseteq U_y$.
This shows that indeed $H' \define H \setminus \{h\} \cup \{y\}$ is a hiding set for~$X$, and since $\conv(\{y,h\}) \cap \conv(X) = \emptyset$, we also have $\card{H'} = \card{H}$.
\end{proof}

With this in mind, we define the \emph{hiding graph}~$G(X) = (V(X), E(X))$ with node set~$V(X) = \obs(X)$, and where two distinct nodes~$x$ and~$y$
are adjacent if and only if~$\conv(\{x,y\}) \cap \conv(X) \neq \emptyset$.
Then, there is a one-to-one correspondence between hiding sets consisting
of observers and cliques in~$G(X)$.
Thus, whenever~$\obs(X)$ is finite, we can compute~$H(X)$ by solving a
maximum size clique problem in~$G(X)$.

Using this graph representation, we can strengthen the hiding
set bound by using the chromatic number~$\chi(G(X))$ of~$G(X)$.
To this end, observe that any relaxation~$Ax \leq b$ of~$X$ defines a
proper coloring of~$G(X)$ by associating with each inequality a unique
color and assigning~$y \in \obs(X)$ the color of one inequality that
separates it from~$X$.
Hence, the chromatic number of~$G(X)$ is a lower bound on~$\rc(X)$.
It is at least as strong as the hiding set bound, because the maximum
size of a clique lower bounds the chromatic number in every
undirected graph.
\begin{remark}
  Replacing~$\obs(X)$ in the definition of~$G(X)$ by a general set~$Y
  \subseteq \Z^d$, the chromatic number of the corresponding graph is a
  lower bound on~$\rc(X, Y)$.
\end{remark}

\section{On the existence of finite certificates}
\label{sec:rc-square}

In this section, we return to the parameter $\rcl(X)$ and, in particular, we focus on the question whether every
full-dimensional lattice-convex set~$X \subseteq \Z^d$ admits a finite
certificate.
The already mentioned identity~$\rc(X) = \rc(X, \obs(X))$ does not provide an answer to
this question, because~$\obs(X)$ can be infinite for~$d\geq 3$
(see~\cite[Sect.~7.5]{weltge2015diss}).
An example for which this happens is~$\Delta_3$, but interestingly there
exists a finite certificate for the fact that $\rc(\Delta_3) \geq 4$.
Indeed, computer experiments led us to a $28$-element set $Y \subseteq \Z^3 \setminus \Delta_3$ such that $\rc(\Delta_3,Y) \ge 4$,
and a computer-based proof can be found online\footnote{\url{https://github.com/christopherhojny/relaxation_complexity/tree/master/computer-aided-proofs}} using SageMath~\cite{sagemath}.
A subgraph of the hiding graph of~$Y$ that certifies this lower bound is given in Figure~\ref{fig:28-certificate-graph};
the integer points corresponding to the nodes in this graph can be found in Table~\ref{tab:certificate}.
The graph has $54$ edges, contains no triangles, is not $3$-colorable, and the removal of any of its edges leads to a graph that can be colored with~$3$ colors.

It turns out that there is nothing special about~$\Delta_3$ concerning such certificates:
By~\cite[Cor.~3.8]{averkovschymura2020complexity}, every full-dimensional $X \subseteq \Z^3$ satisfies $\rc(X) = \rc_\Q(X)$, in particular, \emph{every} relaxation of~$X$ has at least~$4$ facets.
Strengthening this result, we now want to show that there is always a finite certificate for this fact.
We need the notion of \emph{unimodular equivalence} of two sets $X,X' \subseteq \R^d$, which means that there is a unimodular matrix $U \in \Z^{d \times d}$ with $|\det(U)| = 1$ and an integer vector $t \in \Z^d$ such that $X = UX' + t$.

\begin{figure}
\centering
\begin{tikzpicture}[every node/.style={rounded corners,fill=yellow,draw=black},scale=.85]
\node (0) at (-4, -2){0};
\node (1) at (-3, -4){1};
\node (2) at (-4, 0){2};
\node (3) at (3, -1){3};
\node (4) at (3, -3){4};
\node (5) at (-4, -3){5};
\node (6) at (0, -6){6};
\node (7) at (-7, -7){7};
\node (8) at (-1, -6){8};
\node (9) at (4, -3){9};
\node (10) at (1, -6){10};
\node (11) at (7, -7){11};
\node (12) at (0, 0){12};
\node (13) at (-3, -2){13};
\node (14) at (-3, -3){14};
\node (15) at (-2, -5){15};
\node (16) at (-3, -1){16};
\node (17) at (4, 0){17};
\node (18) at (3, -2){18};
\node (19) at (0, 1){19};
\node (20) at (2, -5){20};
\node (21) at (0, 2){21};
\node (22) at (0, 3){22};
\node (23) at (3, -4){23};
\node (24) at (-4, -1){24};
\node (25) at (4, -1){25};
\node (26) at (0, 4){26};
\node (27) at (4, -2){27};

\draw[-] (0) edge (5);
\draw[-] (0) edge (13);
\draw[-] (0) edge (23);
\draw[-] (0) edge (24);
\draw[-] (1) edge (5);
\draw[-] (1) edge (7);
\draw[-] (1) edge (15);
\draw[-] (1) edge (23);
\draw[-] (1) edge (27);
\draw[-] (2) edge (12);
\draw[-] (2) edge (24);
\draw[-] (2) edge (26);
\draw[-] (3) edge (6);
\draw[-] (3) edge (11);
\draw[-] (3) edge (12);
\draw[-] (3) edge (18);
\draw[-] (3) edge (20);
\draw[-] (3) edge (25);
\draw[-] (4) edge (9);
\draw[-] (4) edge (11);
\draw[-] (4) edge (18);
\draw[-] (5) edge (14);
\draw[-] (6) edge (8);
\draw[-] (6) edge (10);
\draw[-] (6) edge (16);
\draw[-] (7) edge (14);
\draw[-] (7) edge (16);
\draw[-] (8) edge (15);
\draw[-] (8) edge (19);
\draw[-] (9) edge (23);
\draw[-] (9) edge (27);
\draw[-] (10) edge (19);
\draw[-] (10) edge (20);
\draw[-] (11) edge (23);
\draw[-] (12) edge (16);
\draw[-] (12) edge (17);
\draw[-] (13) edge (14);
\draw[-] (13) edge (16);
\draw[-] (15) edge (16);
\draw[-] (15) edge (21);
\draw[-] (15) edge (25);
\draw[-] (16) edge (24);
\draw[-] (17) edge (25);
\draw[-] (17) edge (26);
\draw[-] (18) edge (27);
\draw[-] (19) edge (21);
\draw[-] (20) edge (21);
\draw[-] (20) edge (23);
\draw[-] (20) edge (24);
\draw[-] (21) edge (22);
\draw[-] (22) edge (24);
\draw[-] (22) edge (25);
\draw[-] (22) edge (26);
\draw[-] (25) edge (27);
\end{tikzpicture} 
\caption{A non $3$-colorable graph that certifies $\rc(\Delta_3) \geq 4$.}
\label{fig:28-certificate-graph}
\end{figure}
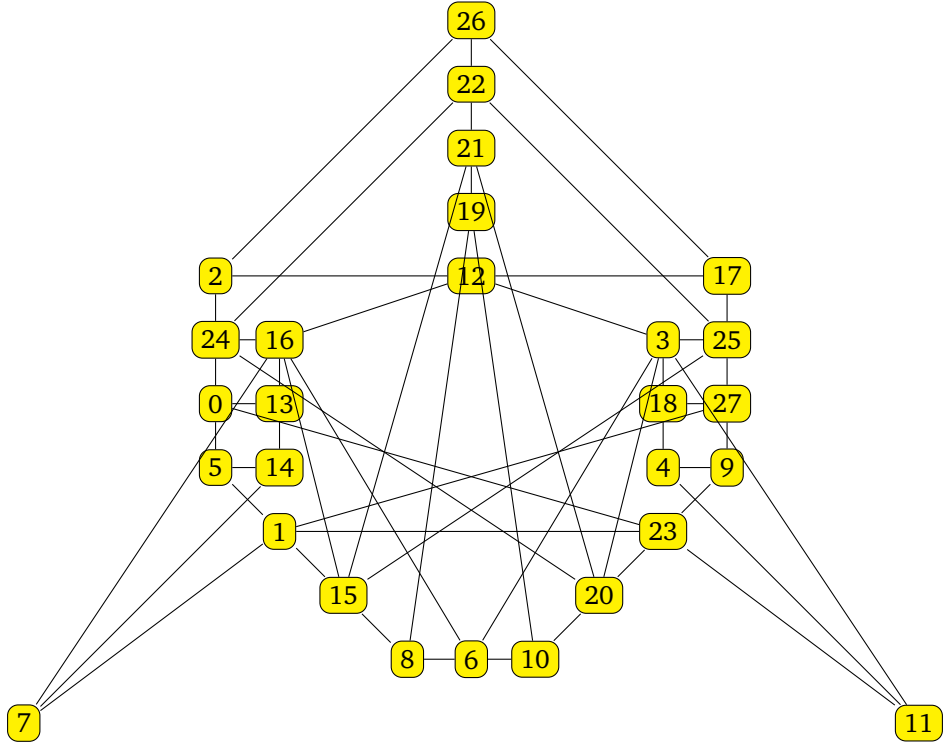

\begin{lemma}
\label{lem:rc-square-dim-3}
Let $X \subseteq \Z^3$ be a lattice-convex set that is a subset of the simplex
\[
T_3 \define \left\{ (x_1,x_2,x_3) \in \R^3 : x_1 \leq 1, x_2 \leq 1, x_3 \leq 1, x_1 + x_2 + x_3 \geq 0 \right\}
\]
and which contains $\Delta_3$.
If $\conv(\Delta_3)$ is a simplex of maximal volume in $\conv(X)$, then $X$ is unimodularly equivalent to one of the four sets corresponding to the columns of the following matrices:
\begin{align}
\left(\begin{smallmatrix*}[r]
0 & 1 & 0 & 0 \\
0 & 0 & 1 & 0 \\
0 & 0 & 0 & 1
\end{smallmatrix*}\right)\ ,\ 
\left(\begin{smallmatrix*}[c]
0 & 1 & 0 & 0 & -1 \\
0 & 0 & 1 & 0 & 0 \\
0 & 0 & 0 & 1 & 1
\end{smallmatrix*}\right)\ ,\ 
\left(\begin{smallmatrix*}[c]
0 & 1 & 0 & 0 & -1 & -1 \\
0 & 0 & 1 & 0 & 0 & 1 \\
0 & 0 & 0 & 1 & 1 & 0
\end{smallmatrix*}\right)\ ,\ 
\left(\begin{smallmatrix*}[c]
0 & 1 & 0 & 0 & -1 & 0 \\
0 & 0 & 1 & 0 & 0 & -1 \\
0 & 0 & 0 & 1 & 1 & 1
\end{smallmatrix*}\right).\label{eq:rc-square-dim-3}
\end{align}
Moreover, all these four sets satisfy $\rcl(X) = \rc(X) = 4$.
\end{lemma}

\begin{proof}
The $20$ lattice points in the tetrahedron~$T_3$ correspond to the columns of the following matrix:
\[
\left(\begin{smallmatrix*}[c]
0 & 1 & 0 & 0 & -2 & -1 & -1 & -1 & 0 & 0 & 0 & 1 & 1 & 1 & 1 & 1 & 1 & 1 & 1 & 1 \\
0 & 0 & 1 & 0 & 1 & 0 & 1 & 1 & -1 & 1 & 1 & -2 & -1 & -1 & 0 & 0 & 1 & 1 & 1 & 1 \\
0 & 0 & 0 & 1 & 1 & 1 & 0 & 1 & 1 & -1 & 1 & 1 & 0 & 1 & -1 & 1 & -2 & -1 & 0 & 1
\end{smallmatrix*}\right).
\]
Every lattice-convex set $X \subseteq T_3$ such that $\conv(X)$ contains $\conv(\Delta_3)$ as a simplex of maximal volume corresponds to a subset~$S$ of the last $16$ columns of that matrix such that every tetrahedron with vertices in $\Delta_3 \cup S$ has volume $1/6$.
Using an implementation\footnote{\url{https://github.com/christopherhojny/relaxation_complexity/blob/master/computer-aided-proofs/3d-unimodular-sets.sage}} in SageMath~\cite{sagemath}, we find that every such subset $S$ can contain at most~$2$ elements.
This leaves us with $10 + 30 = 40$ candidates for~$S$, each of which corresponding to a lattice polytope~$P_S = \conv(\Delta_3 \cup S)$ with $5$ or $6$ vertices, respectively.
Finally, we use that two of the polytopes $P_S$, $P_{S'}$ are unimodularly equivalent if and only if their \emph{affine normal form} coincides.
We refer to~\cite[Sect.~4]{averkovborgersoprunov2020classification} for the definition of this concept and a proof of this statement.
The authors of that paper also provide an implementation in SageMath of the affine normal form which is available online\footnote{\url{https://github.com/christopherborger/mixed_volume_classification/blob/master/polytopes.sage}}.

The fact that all the four thus obtained sets $X$ satisfy $\rcl(X) = \rc(X) = 4$ can be certified by (a) corresponding finite certificates~$Y$ that prove the necessity of four facets in each relaxation, and (b) an explicit relaxation of~$X$ by a tetrahedron.
Just as for the standard simplex $\Delta_3$, our implementation for
$\rc(X,Y)$ allows us to do just that (see \url{https://github.com/christopherhojny/relaxation_complexity/blob/master/computer-aided-proofs/verify_rc_representatives.sage} for the precise data).
\end{proof}

\begin{table}
  \centering
  \caption{A finite certificate for $\rcl(\Delta_3) \geq 4$.}
  \label{tab:certificate}
  \begin{tabular*}{\textwidth}{@{}lrrr@{\;\;\;\;\;\extracolsep{\fill}}rrrrr@{\;\;\;\;\;\extracolsep{\fill}}rrrrr@{\;\;\;\;\;\extracolsep{\fill}}rrrrr@{}}\toprule
    ID & \multicolumn{3}{c}{coordinates} && ID & \multicolumn{3}{c}{coordinates} && ID & \multicolumn{3}{c}{coordinates} && ID & \multicolumn{3}{c}{coordinates}\\
    \midrule
    0  &  0 & 1 &  1 &&  7 & -1 &  0 &  0 && 14 &  1 &  1 &  1 && 21 &  2 & -1 &  0\\
    1  &  1 & 0 &  1 &&  8 &  2 &  0 &  0 && 15 & -1 &  1 &  0 && 22 & -1 &  2 &  0\\
    2  &  0 & 2 &  0 &&  9 &  0 &  0 &  2 && 16 &  1 &  1 &  0 && 23 &  1 &  0 & -1\\ 
    3  &  1 & 1 & -1 && 10 &  2 &  0 & -1 && 17 &  0 &  2 & -1 && 24 &  1 & -1 &  0\\
    4  &  1 & 1 & -2 && 11 & -1 &  0 &  2 && 18 &  0 & -1 &  2 && 25 &  1 & -1 &  1\\
    5  &  0 & 0 & -1 && 12 &  0 & -1 &  1 && 19 & -2 &  1 &  1 && 26 &  1 & -2 &  1\\
    6  & -1 & 0 &  1 && 13 &  0 & -1 &  0 && 20 & -1 &  1 &  1 && 27 &  0 &  1 & -1\\
    \bottomrule
  \end{tabular*}
\end{table}

\begin{proposition}
\label{prop:finite-certificate-dim3}
Any full-dimensional lattice-convex $X \subseteq \Z^3$ satisfies $\rcl(X) \geq 4$.
\end{proposition}

\begin{proof}
Let $b_0,b_1,b_2,b_3 \in X$ be the vertices of a simplex $S_X \define \conv(\{b_0,b_1,b_2,b_3\})$ of maximal volume in $\conv(X)$.
We may assume w.l.o.g.\ that $b_0 = \zero$ and we write $B=(b_1,b_2,b_3) \in \Z^{3 \times 3}$, so that $S_X = B \conv(\Delta_3)$.
It is a well-known result that as $S_X$ is volume-maximal, we have $\conv(X) \subseteq B T_3$, where $T_3$ is the simplex in Lemma~\ref{lem:rc-square-dim-3} (cf.~\cite[Thm.~3]{lagariasziegler1991bounds}).
Now, let $\Lambda = B \Z^3 \subseteq \Z^3$ be the lattice spanned by $b_1,b_2,b_3$.
In view of Lemma~\ref{lem:rc-square-dim-3}, $\Lambda \cap X = B X'$, where $X'$ is unimodularly equivalent to one of the four sets consisting of the columns of either of the matrices in~\eqref{eq:rc-square-dim-3}.

Now, if $Q$ is a relaxation of $X$ within $\Z^3$, then $Q \cap \Lambda = (Q \cap \Z^3) \cap \Lambda = X \cap \Lambda = B X'$.
Thus, $Q$ is a relaxation of $B X'$ within $\Lambda$.
This implies the relation $\rcl(X) \geq \rcl(BX',\Lambda) = \rcl(X',\Z^3) =
4$ as desired.
\end{proof}

Note that this result does not provide us with a finite certificate for the \emph{exact} value of $\rc(X)$, for a given set $X \subseteq \Z^3$.
However, we conjecture that this is always possible, at least in three dimensions:

\begin{conjecture}
Every full-dimensional lattice-convex set $X \subseteq \Z^3$ admits a finite set $Y \subseteq \Z^3 \setminus X$ such that $\rc(X) = \rc(X,Y)$.
\end{conjecture}

A curious phenomenon is that, in contrast to Proposition~\ref{prop:finite-certificate-dim3}, in any higher dimension $d \geq 4$, the possible necessity of $d+1$ facets in a relaxation of $\Delta_d$ cannot be certified by a finite subset $Y \subseteq \Z^d \setminus \Delta_d$.
In particular, there is no finite certificate for the fact that $\rc(\Delta_4)=5$, giving a negative answer to a question of Weltge~\cite[Problem~12]{weltge2015diss} and showing that $\rcl(\Delta_4) < \rc(\Delta_4)$ (cf.~Theorem~\ref{thm:epsrelaxProperties}).

\begin{theorem}
  \label{thm:rc-square-simplex}
For every $d \in \Z_{>0}$ holds $\rcl(\Delta_d) \le \ceiling{\frac{d}{2}}+2$.
\end{theorem}

\begin{proof}
  The assertion is true for $d \le 3$, because $\rcl(\Delta_d)
  \leq \rc(\Delta_d) \leq d+1$ and $d+1 \leq \ceiling{\frac{d}{2}}+2$
  in these cases.
  Thus, we may assume that $d \ge 4$.
  Moreover, if the result holds for even dimensions, we can easily derive it for odd~$d$, because
  $\rcl(\Delta_d) \le \rcl(\Delta_{d+1})$, since $\Delta_d =
  \setcond{(x_1,\ldots,x_{d+1}) \in \Delta_{d+1}}{x_{d+1} = 0}$ and
  $\rcl(\Delta_{d+1}) \leq \frac{d+1}{2} + 2 = \lceil\frac{d}{2}\rceil + 2$
  for even~$d+1$.

  For this reason, suppose~$d$ is even.
  Let $k = \frac{d}{2} +1$. Let $N \in \Z_{>0}$ be arbitrary and let $\varepsilon
  = \frac{1}{3 N}$.
  Moreover, let~$\delta_1, \dots, \delta_{k-2} \in [0,1] \setminus \Q$ be
  such that $1,\delta_1, \dots, \delta_{k-2}$ are linearly independent over~$\Q$ and let $\mathcal{Z} \define \Z^k \times \{-N,\ldots,N\}^{k-2}
  \subseteq \Z^d$.
  Consider the full-dimensional simplex
  \begin{equation}
    \label{eq:rcboxsimplexA}
    Q
    =
    \conv \Big\{
    \tfrac{1}{1 - \varepsilon} (e_1 + e_k), \tfrac{1}{1 - \varepsilon}e_k,
    e_2, \dots, e_{k-1}, \zero
    \Big\} \subseteq \R^k
  \end{equation}
  and the linear map~$\pi\colon \R^d \to \R^k$, $z \mapsto y$, with
  \begin{subequations}
    \label{eq:rcboxsimplexPartB}
    \begin{numcases}{y_i =}
      z_1 + \sum\nolimits_{j = 1}^{k-2} \delta_j z_{k+j}, & if $i = 1$,\label{eq:rcboxsimplexB}\\
      z_i + \varepsilon z_{k+i-1}, & if $i \in \{2,\dots,k-1\}$,\label{eq:rcboxsimplexC}\\
      z_1 + \sum\nolimits_{j = k}^{d} z_j, & if $i = k$.\label{eq:rcboxsimplexD}
    \end{numcases}
  \end{subequations}
  In the following, we show that~$P = \{ x \in \R^d : \pi(x) \in Q \}$ is
  a relaxation of~$\Delta_d$ within $\mathcal{Z}$, that is, $P \cap
  \mathcal{Z} = \Delta_d$.
  Since~$Q$ has~$k+1$ facets, this shows the relation $\rcl(\Delta_d) \leq \rc(\Delta_d, \mathcal{Z}) \leq k+1 = \frac{d}{2}
  + 2$ as desired.

  First, we show~$\pi(\Delta_d) \subseteq Q$.
  To this end, note that~$\pi(\zero) = \zero$ and~$\pi(e_i) = e_i$
  for~$i \in \{2,\dots,k-1\}$.
  Moreover, $\pi(e_1) = e_1 + e_k$ and~$\pi(e_k) = e_k$, which are
  clearly contained in~$Q$ since~$1 < \frac{1}{1 - \varepsilon}$.
  For $i \in \{k+1,\dots,d\}$, $\pi(e_i) = \delta_{i - k} e_{1} +
  \varepsilon e_{i - k + 1} + e_k$.
  Since~$\pi(e_i)$ is the convex combination
  of~$\frac{1}{1-\varepsilon}(e_1 + e_k)$, $e_{i-k+1}$,
  and~$\frac{1}{1-\varepsilon}e_k$ with coefficients
  $(1-\varepsilon)\delta_{i-k}$,
  $\varepsilon$, and
  $(1 - \varepsilon)(1 - \delta_{i-k})$, respectively,
  we conclude the first part of the proof because~$\pi(\Delta_d) \subseteq Q$.

  In the second part, we show that any~$z \in \mathcal{Z}$ with~$\pi(z) \in
  Q$ is necessarily contained in~$\Delta_d$.
  Then, the theorem's statement follows, because~$N$ can be selected arbitrarily.
  To show that~\eqref{eq:rcboxsimplexA}, \eqref{eq:rcboxsimplexPartB},
  and~$z \in \mathcal{Z}$ imply~$z \in \Delta_d$, we first of all note
  that~\eqref{eq:rcboxsimplexD} and~$z \in \mathcal{Z} \subseteq \Z^d$ yield~$y_k \in
  \Z$.
  For this reason, we have $0 \leq y_k \leq \frac{1}{1 - \varepsilon} = \frac{3N}{3N-1} < 2$
  by~\eqref{eq:rcboxsimplexA}.
  That is, $y_k \in \{0,1\}$.
  We proceed with a case distinction on the binary value of~$y_k$.

  On the one hand, if~$y_k = 0$, then~\eqref{eq:rcboxsimplexA} shows~$y \in
  \conv\{e_2, \dots, e_{k-1}, \zero\}$, i.e., $y_1 = 0$.
  As~$1,\delta_1, \dots, \delta_{k-2}$ are linearly independent over~$\Q$,
  we derive from~\eqref{eq:rcboxsimplexB} that~$z_1 = z_{k+1} = \dots = z_d
  = 0$.
  Then, we conclude
  $
  0
  =
  y_k
  =
  z_1 + \sum\nolimits_{j = k}^d z_j
  =
  z_k
  $
  by~\eqref{eq:rcboxsimplexD},
  and~$y_i = z_i$ for all~$i \in \{2,\dots,k-1\}$ by~\eqref{eq:rcboxsimplexC}.
  Consequently, the latter shows $(z_2,\dots,z_{k-1}) \in \Delta_{k-2}$
  as~$y \in \conv\{e_2, \dots, e_{k-1}, \zero\}$.
  Since the remaining entries of~$z$ equal~$0$, we get $z \in \Delta_{d}$.

  On the other hand, suppose~$y_k = 1$.
  Since~$Q$ is the convex hull of~$Q_1 = \frac{1}{1-\varepsilon} \conv
  \{e_1 + e_k, e_k\}$ and $Q_2 =  \conv \{e_2, \dots, e_{k-1}, \zero\}$, the
  cross-section of~$Q$ at height~$y_k = 1$ is~$(1 - \varepsilon)Q_1 +
  \varepsilon Q_2$.
  That is, any~$y$ satisfying~\eqref{eq:rcboxsimplexA},
  \eqref{eq:rcboxsimplexPartB}, $z \in \mathcal{Z}$, and~$y_k = 1$ also
  satisfies
  \begin{align}
    \label{eq:conclusionBoxPartB}
    0 \leq y_1 \leq 1,
    &&
       y_2, \dots, y_{k-1} \geq 0,
    &&
       y_2 +  \dots + y_{k-1} \leq \varepsilon,
    &&
       y_k = 1.
  \end{align}
  From~$z \in \mathcal{Z}$, we get, for all~$i \in \{2, \dots, k-1\}$,
  \mbox{$-\frac{1}{3} = -\varepsilon N \leq \varepsilon z_{k+i-1} \leq
  \varepsilon N = \frac{1}{3}$.}
  Thus, as~$z_i \in \Z$, we conclude from~\eqref{eq:rcboxsimplexC} that
  $y_i < 0$ or~$y_i > \varepsilon$ if~$z_i \neq 0$.
  Consequently, for every~$i \in \{2, \dots, k-1\}$,
  Properties~\eqref{eq:rcboxsimplexC} and \eqref{eq:conclusionBoxPartB} imply $z_i =
  0$ and $\tilde{z} = (z_{k+1}, \dots, z_{d}) \in \Delta_{k-2}$.
  The only possible non-zero entries of~$z$ are thus~$z_1$, $z_k$, and
  possibly one further entry~$z_j$, for some~$j \in \{k+1,\dots,d\}$.
  To conclude the proof, we distinguish the different cases
  for~$\tilde{z}$.

  As~$\tilde{z} \in \Delta_{k-2}$, either~$\tilde{z} = \zero$ or there exists
  exactly one~$j \in \{2,\dots, k-1\}$ with~$\tilde{z}_{j-1} = z_{k+j-1} = 1$
  and the remaining entries of~$\tilde{z}$ are~0.
  In the first case, \eqref{eq:rcboxsimplexB} shows~$y_1 = z_1$.
  From~$z \in \mathcal{Z}$ and~$0 \leq y_1 \leq 1$, we get~$y_1 \in
  \{0,1\}$.
  If~$z_1 = 0$, then~\eqref{eq:rcboxsimplexD} and~$\tilde{z} = \zero$
  imply~$y_k = z_k$.
  That is, $z \in \{\zero, e_k\} \subseteq \Delta_d$.
  If~$z_1 = 1$, then $1 = y_k
  \overset{\eqref{eq:rcboxsimplexD}}{=} z_1 + z_k = 1 + z_k$ shows~$z_k =
  0$, i.e., $z = e_1 \in \Delta_d$.
  In the second case, let~$j$ be the unique index in $\{2,\dots, k-1\}$
  with~$z_{k+j-1} = 1$.
  Then, we can derive from~\eqref{eq:rcboxsimplexB} that
  $[0,1] \ni y_1 = z_1 + \delta_{j-1} z_{k+j-1} = z_1 + \delta_{j-1}$.
  Since~$z_1 \in \Z$ and~$\delta_{j-1} \in (0,1)$, we conclude~$z_1 = 0$.
  From~\eqref{eq:rcboxsimplexD}, we thus get
  \[
    1
    =
    y_k
    =
    z_1 + \sum_{i = k}^d z_i
    =
    z_k + z_{k+j-1}
    =
    z_k + 1.
  \]
  Consequently, $z_k = 0$ and hence~$z = e_{k+j-1} \in \Delta_d$, finishing the
  proof.
\end{proof} 

Now, we aim to complement this upper bound on $\rcl(\Delta_d)$ by providing a lower bound that grows with the dimension~$d$.
This is in analogy to the lower bound $\rc(X) > \log_2(\dim(X)) - \log_2 \log_2(\dim(X))$, for any $X \subseteq \Z^d$ with $\dim(X) \geq 4$ proven in~\cite[Thm.~1.2]{averkovschymura2020complexity}.
To this end, we need a preparatory combinatorial lemma.

We call a tuple $(\alpha_1,\ldots,\alpha_n) \in \R^n$ \emph{monotonic} if $\alpha_1 \ge \cdots \ge \alpha_n$ or $\alpha_1 \le \cdots \le \alpha_n$ holds. The following is a special case of  De Bruijn's theorem from combinatorics. Since the result might not be well-known in integer optimization, we reproduce its short proof. 

\begin{lemma} \label{lem:debruijn}
	Let $A \in \R^{m \times n}$ be a matrix with $n = 2^{2^m} +1$ columns.
	Then,~$A$ has an $(m \times 3)$-submatrix whose every row is monotonic. 
\end{lemma}

\begin{proof} 
		We reproduce the argument given in \cite[Sect.~5]{kruskal1953monotonic} and \cite[Thm.~2]{spencer1971minimal}. 
		We use the celebrated Erd\H{o}s-Szekeres theorem (see~\cite{erdosszekeres1935acombinatorial}) which asserts that each $(N^2+1)$-term  sequence of real numbers has an $(N+1)$-term monotonic subsequence. 
	
		We argue by induction on $m$. For $m=1$, the assertion follows from the Erd\H{o}s-Szekeres theorem with $N=2$. Assume $m \ge 2$ and that the assertion has been verified for matrices with $m-1$ rows. 
		Applying the Erd\H{o}s-Szekeres theorem  with $N=2^{2^{m-1}}$, we see 
			that the matrix $A \in \R^{m \times n}$ has a submatrix $B$ of dimension $m \times (2^{2^{m-1}} + 1)$, whose first row is monotonic. Applying the induction assumption to the matrix obtained from $B$ by removing the first row, we obtain that $B$ has an \mbox{$(m-1) \times 3$}-submatrix~$B'$ whose every row is monotonic.
Amending this submatrix with the three elements of the first row of~$B$ that belong to the three columns of~$B'$, we get the desired monotonic $(m \times 3)$-submatrix of~$A$.
\end{proof} 

Complementing Theorem~\ref{thm:rc-square-simplex}, we can now prove that $\rcl(\Delta_d) \to \infty$ as $d \to \infty$:

\begin{theorem}
	\label{thm:rc:squar:delta:d}
	For each $d \ge 4$, one has $\rcl(\Delta_d) \ge \log _2\log_2 d$. 
\end{theorem} 

\begin{proof} 
	Let $Y \define \setcond{ e_j - e_k + e_\ell}{ 1 \le j < k < \ell
          \le d}$. Since $\rcl(\Delta_d) \ge \rc(\Delta_d, Y)$, it suffices
        to show $\rc(\Delta_d, Y) \ge \log_2 \log_2 d$.  Consider a system
        of~$m$ linear functionals $U x + v$ such that
        \[
          \setcond{ x \in \Delta_d \cup Y }{U x + v \ge 0} = \Delta_d.
        \]
        Taking into account
        that $\one^\intercal x = 1$, for each $x \in \{e_1,\ldots,e_d\}
        \cup Y$, we can see that the $(m \times d)$-matrix 
	\(
		A \define (a_{i,j}) \define U + v \one^\intercal
	\) satisfies 
	\begin{equation} \label{std:vectors:and:Y}
			\setcond{x \in \{e_1,\ldots,e_d\} \cup Y }{A x \ge 0} = \{e_1,\ldots,e_d\}.
	\end{equation}
 We claim that $m \ge \log_2 \log_2 d$. If this were not the case, then we had $d \ge 2^{2^m} + 1$. Hence, Lemma~\ref{lem:debruijn} would imply that $A$ has an $(m \times 3)$-submatrix whose every row is monotonic. Let $1 \le j < k < \ell \le d$ be the indices of the columns of  $A$ that induce such a submatrix. By construction, $A e_1 \ge 0,\ldots, A e_d \ge 0$, which means that all entries  of $A$ are non-negative. For $j,k,\ell$ as chosen above, $A (e_j  - e_k + e_\ell ) \ge 0$ holds. Indeed, for each $i \in [m]$, one has $a_{i,j} \ge a_{i,k} \ge a_{i,\ell} \ge 0$ or $0 \le a_{i,j} \le a_{i,k} \le a_{i,\ell}$. Independently of which of the two cases occurs, one always has $a_{i,j} - a_{i,k} + a_{i,\ell} \ge 0$. We have thus found the  point $y \define e_j - e_k + e_\ell \in Y$ that satisfies $A y \ge 0$, which is a contradiction to \eqref{std:vectors:and:Y}. This shows $m \ge \log_2 \log_2 d$ and concludes the proof.
\end{proof} 

\begin{remark}
  \label{rem:correct-order-Y}
  Since the growth of the doubly logarithmic bound in Theorem~\ref{thm:rc:squar:delta:d} is very
  slow, it is natural to wonder if this bound can be improved.
  It turns out that our way of estimating $\rcl(\Delta_d)$ via
  $\rc(\Delta_d, Y)$, with the particular choice of $Y$ fixed in the proof, cannot give any better estimate, since $\rc(\Delta_d, Y)$ is of order $\log
  \log d$ up to an absolute multiplicative constant.
  The upper bound on $\rc(\Delta_d,Y)$ has been clarified in an email
  communication with M.\ Buci\'{c} and B.\ Sudakov.
\end{remark}

There are two natural open problems related to Theorem~\ref{thm:rc:squar:delta:d}. The first one is whether the assertion can be extended to $\rcl(X) \ge \log_2 \log_2 d$ for an arbitrary $d$-dimensional lattice-convex set~$X$. The second is whether the lower bound can be strengthened.

\section{Computational Complexity and Hiding Sets in Dimension 2}
\label{sec:AlgorithmDimension2}

In this section, we confine ourselves with the special case of two-dimensional lattice-convex sets \mbox{$X \subseteq \Z^2$}.
We show that the relaxation complexity of any such~$X$ can be computed in weakly polynomial time and that the size $H(X)$ of a maximal hiding set deviates by at most~$1$ from~$\rc(X)$.

The main property of two-dimensional lattice-convex sets that leads to these results is that they always have only finitely many observers, a fact that Weltge~\cite{weltge2015diss} already realized.
Since the set of observers is a certificate for $\rc(X)$, this means that
the inequalities in~\eqref{eqn:inequalities-bounds} are all in fact
identities for $X \subseteq \Z^2$ (since $\rc(X) = \rc_\Q(X)$ if~$d=2$,
which is a consequence of~\cite[Sec.~7.5]{weltge2015diss}) and that the relaxation complexity is computable.

In order to devise an efficient algorithm to compute $\rc(X)$, we first interpret the relaxation complexity of any lattice-convex set~$X \subseteq \Z^d$, such that~$\obs(X)$ is finite, as the size of a minimal covering of the observers.
More precisely, the separation theorems for convex sets imply that
\[
\mathcal{I}(X) \define \left\{ I \subseteq \obs(X) \st \conv(X) \cap
\conv(I) = \emptyset \right\}
\]
contains all subsets of observers that can be
separated from~$X$ by a single hyperplane.
Let~$\mathcal{I}_{\max}(X)$ be the family of all inclusion-wise
maximal sets in~$\mathcal{I}(X)$.

\begin{observation}
  \label{obs:RCintervalcovering}
  If~$X \subseteq \Z^d$ is such that $\obs(X)$ is finite, then $\rc(X)$ is the smallest number~$k$ of sets~$I_1, \dots,
  I_k \in \mathcal{I}_{\max}(X)$ such that~$\obs(X) = \bigcup_{i = 1}^k I_i$.
\end{observation}

In order to computationally make use of this observation, a crucial step is to find the set $\obs(X)$ of observers of any given two-dimensional lattice-convex set~$X \subseteq \Z^2$.
This can be done efficiently because $\obs(X)$ is
the set of integer points in the boundary of enlarging the lattice polygon
$P=\conv(X)$ by lattice-distance one over each of its
edges (see Weltge~\cite[Prop.~7.5.6]{weltge2015diss}).

To make this precise, let $P = \{x \in \R^2 : a_i x_1 + b_i x_2 \leq c_i,
 i \in [k]\}$ be a polyhedron described by~$k$ inequalities with the
assumption that $a_i, b_i, c_i \in \Z$ with $a_i$ and $b_i$ coprime, for all $i \in [k]$. 
Then, following Castryck~\cite{castryck2012movingout}, we write
$
P^{(-1)} \define \left\{ x \in \R^2 : a_i x_1 + b_i x_2 \leq c_i + 1,\; i \in [k] \right\}
$
and say that $P^{(-1)}$ is obtained from~$P$ by \emph{moving out the edges}.
Note that neither does $P^{(-1)}$ need to be a lattice polygon again, nor
may it have as many edges as~$P$.
With this notation the previous discussion can be formulated as
\[
\obs(X) = \bd(P^{(-1)}) \cap \Z^2.
\]
In particular, this means $Y = \obs(X)$ is in convex position, i.e., $Y \subseteq \bd(\conv(Y))$, and that we can efficiently list its elements in counterclockwise order.

\begin{lemma}
\label{lem:determine-observers-plane}
Let~$V \subseteq \Z^2$ be a finite two-dimensional set and let $X = \conv(V) \cap \Z^2$.
There is an algorithm that determines $\obs(X) =
\{y_0,\ldots,y_\ell\}$, with the labeling in counterclockwise order, and which runs in~$\cO(\ell + k \log k + \gamma k)$ time, where $k = \card{V}$ and $\gamma$ is an upper bound on the binary encoding size of any point in~$V$.
\end{lemma}

\begin{proof}
Let $P=\conv(V)$.
The algorithm consists of the following four steps:
\begin{enumerate}
 \item Compute an irredundant inequality description $P = \{x \in \R^2 :
   a_i x_1 + b_i x_2 \leq c_i, i \in [k]_0\}$, where $a_i,b_i,c_i \in \Z$ with $a_i$ and $b_i$ coprime, and the outer normal vectors $(a_i,b_i)^\intercal$ labeled in counterclockwise order.
 \item Move out the edges of $P$ and let $P^{(-1)} = \{x \in \R^2 : a_{i_j}
   x_1 + b_{i_j} x_2 \leq c_{i_j}+1, j \in [m]_0\}$ be such that all redundancies are removed.
 \item Compute the set of vertices $\{w_0,w_1,\ldots,w_m\}$ of $P^{(-1)}$ in counterclockwise order.
 \item For each $i \in [m]_0$, compute the integer points on the segment $[w_i,w_{i+1}]$, with the index taken modulo $m+1$.
\end{enumerate}
Some detailed comments are in order:

In Step~1, we first use a standard convex hull algorithm in the plane, e.g., Graham's scan (cf.~\cite[Ch.~8]{edelsbrunner1987algorithms} for details) to compute the vertices $\{v_0,\ldots,v_r\}$ of~$P$ in counterclockwise order.
Since~$P$ is a lattice polygon, each of its edges, say $[v_j,v_{j+1}]$, corresponds to an integer vector $(\eta_1,\eta_2)^\intercal = v_{j+1}-v_j$.
Because of the counterclockwise ordering, $(\eta_2,-\eta_1)^\intercal$ is an outer normal vector to the edge at hand, and dividing out by the greatest common divisor of $\eta_1$ and $\eta_2$ leads to the desired inequality description.
When we use Euclid's Algorithm in this last step, we obtain a running time of $\cO(k \log k + \gamma r) \subseteq \cO(k \log k + \gamma k)$.

The only thing to do in Step~2, besides increasing all the right hand sides by one unit, is to remove the redundancies.
One way to do this is to use duality between convex hulls and intersections of hyperplanes, and again invoke, e.g., Graham's scan.
This can be done in time $\cO(k \log k)$.

Step~3 just amounts to an iterative computation of the intersection point $w_j$
of the pair of equations $a_{i_j} x_1 + b_{i_j} x_2 = c_{i_j}+1$ and
$a_{i_{j+1}} x_1 + b_{i_{j+1}} x_2 = c_{i_{j+1}}+1$, for $j \in [m]_0$.
In this computation we also record the normal vector $(a_{i_{j+1}},b_{i_{j+1}})^\intercal$ that corresponds to the edge with endpoints~$w_j$ and~$w_{j+1}$.
This needs $\cO(k)$ steps.

For Step~4 we may use Euclid's Algorithm on the defining data of the edge of $P^{(-1)}$ that contains $w_i$ and $w_{i+1}$, and determine an affine unimodular transformation $A_i\colon \R^2 \to \R^2$ such that $A_i [w_i,w_{i+1}] = [\omega_i e_1, \omega_{i+1} e_1]$, for some $\omega_i < \omega_{i+1}$.
The integer points on the latter segment in increasing order of the first coordinate are given by
\[
z_j = \frac{\omega_{i+1}-\lceil \omega_i \rceil - j}{\omega_{i+1} - \omega_i} \omega_i e_1 + \frac{\lceil \omega_i \rceil + j - \omega_i}{\omega_{i+1} - \omega_i} \omega_{i+1} e_1, \text{ for }j\in\{0,\ldots,\lfloor \omega_{i+1} \rfloor - \lceil \omega_i \rceil\}.
\]
Using the inverse transformation~$A_i^{-1}$ then leads to the correctly ordered list of integer points on the segment $[w_i,w_{i+1}]$.
For a given edge of $P^{(-1)}$ these steps can be performed in time proportional to the number of integer points that it contains, leading to a total running time of $\cO(\ell + \gamma k)$ for this step.

Conclusively, we saw that the outlined algorithm terminates with the correctly computed list of observers after $\cO(\ell + k \log k + \gamma k)$ iterations.
\end{proof}

We now put things together and show that~$\rc(X)$ can be computed
efficiently.

\begin{theorem}
\label{thm:rc-plane-alg}
  Let~$V \subseteq \Z^2$ be a finite two-dimensional set, let $X = \conv(V) \cap \Z^2$, and let~$Y = \obs(X)$.
  Then, the relaxation complexity $\rc(X)$ can be computed in
  time $\cO(\card{V} \cdot \log\card{V} + \card{V} \cdot \card{Y} \cdot \log\card{Y} + \gamma \cdot \card{V})$, where $\gamma$ is an upper bound on the binary encoding size of any point in~$V$.
\end{theorem}

\begin{proof}
  Assume that the set of observers~$\obs(X)$ is given in
  counterclockwise order~$y_0,\dots,y_\ell$.
  Then, for each~$I \in \mathcal{I}_{\max}(X)$, there exist~$r,s \in [\ell]_0$ such
  that $I = \{y_r,y_{r+1},\dots,y_{r+s}\}$, where indices are modulo~$\ell +
  1$.
  That is, the sets in~$\mathcal{I}_{\max}(X)$ form ``discrete intervals'' of
  observers, see Figure~\ref{fig:intervals}.
  In particular, $\mathcal{I}_{\max}(X)$ contains at most~$\ell + 1$
  intervals.

  Because of Observation~\ref{obs:RCintervalcovering}, we can determine~$\rc(X)$
  by finding the smallest number of intervals
  in~$\mathcal{I}_{\max}(X)$ that is sufficient to cover~$\obs(X)$.
  This problem can be solved in~$\cO(\ell \log \ell)$ time using the minimum
  circle-covering algorithm by Lee \& Lee~\cite{LeeLee1984}.
  Thus, the assertion follows if the observers and the
  intervals in~$\mathcal{I}_{\max}(X)$ can be computed in
  $\cO(\card{V} \cdot \log\card{V} + \card{V} \cdot \card{Y} \cdot
  \log\card{Y} + \gamma \cdot \card{V})$ time.

  By Lemma~\ref{lem:determine-observers-plane}, the list of observers $\obs(X)$ in counterclockwise order can be found
  in $\cO(\card{Y} + \card{V} \cdot \log\card{V} + \gamma \cdot \card{V})$ time.
  To find the sets~$I \in \mathcal{I}_{\max}(X)$, note that we can use binary search on~$y_{r+s}$
  to find a maximum interval~$\{y_r,\dots,y_{r + s}\}$ such
  that~$\conv(\{y_r,y_{r+s}\}) \cap \conv(V) = \emptyset$.
  In each of the~$\cO(\log\card{Y})$ steps of the binary search, we have to check whether the line segment~$\conv(\{y_r, y_{r+s}\})$ intersects one of the~$\cO(\card{V})$ edges of~$\conv(V)$.
  Thus, a single set~$I \in \mathcal{I}_{\max}$ can be computed in~$\cO(\card{V} \cdot
  \log\card{Y})$ time.
  Combining these running times concludes the proof.
\end{proof}

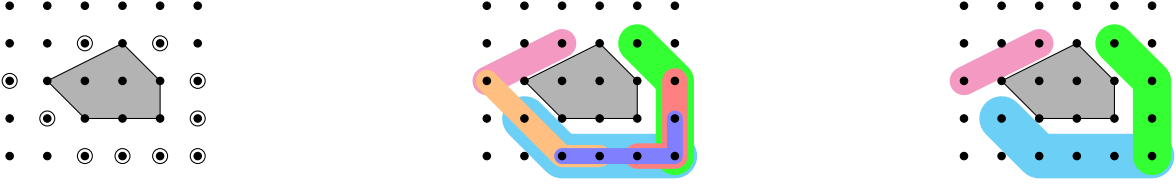
\begin{figure}[t]
  \centering
  \begin{tikzpicture}[scale=0.5]
    \draw[fill=black!30] (2,1) -- (4,1) -- (4,2) -- (3,3) -- (1,2) -- (2,1);
    \draw[line width=3.5ex,draw=white,line cap=round,line join=round] (5,0) -- (2,0) --(1,1);

    \foreach \x in {0,...,5}
    {
      \foreach \y in {0,...,4}
      {
        \node (\x\y) at (\x,\y) [circle,draw=black,fill=black,inner sep=1pt] {};
      }
    }

    \node (a) at (2,3) [circle,draw=black,inner sep=2pt] {};
    \node (b) at (0,2) [circle,draw=black,inner sep=2pt] {};
    \node (c) at (1,1) [circle,draw=black,inner sep=2pt] {};
    \node (d) at (2,0) [circle,draw=black,inner sep=2pt] {};
    \node (e) at (3,0) [circle,draw=black,inner sep=2pt] {};
    \node (f) at (4,0) [circle,draw=black,inner sep=2pt] {};
    \node (g) at (5,0) [circle,draw=black,inner sep=2pt] {};
    \node (h) at (5,1) [circle,draw=black,inner sep=2pt] {};
    \node (i) at (5,2) [circle,draw=black,inner sep=2pt] {};
    \node (j) at (4,3) [circle,draw=black,inner sep=2pt] {};

  \end{tikzpicture}
  \hfill
  \begin{tikzpicture}[scale=0.5]
    \draw[fill=black!30] (2,1) -- (4,1) -- (4,2) -- (3,3) -- (1,2) -- (2,1);

    \draw[line width=3.5ex,draw=cyan!50,line cap=round,line join=round] (5,0) -- (2,0) --(1,1);
    \draw[line width=2.25ex,draw=magenta!50,line cap=round,line join=round] (0,2) -- (2,3);
    \draw[line width=1.75ex,draw=orange!50,line cap=round,line join=round] (3,0) -- (2,0) -- (0,2);
    \draw[line width=3ex,draw=green!80,line cap=round,line join=round] (5,0) -- (5,2) -- (4,3);
    \draw[line width=2ex,draw=red!50,line cap=round,line join=round] (4,0) -- (5,0) -- (5,2);
    \draw[line width=1.25ex,draw=blue!50,line cap=round,line join=round] (2,0) -- (5,0) -- (5,1);

    \foreach \x in {0,...,5}
    {
      \foreach \y in {0,...,4}
      {
        \node (\x\y) at (\x,\y) [circle,draw=black,fill=black,inner sep=1pt] {};
      }
    }

  \end{tikzpicture}
  \hfill
  \begin{tikzpicture}[scale=0.5]
    \draw[fill=black!30] (2,1) -- (4,1) -- (4,2) -- (3,3) -- (1,2) -- (2,1);

    \draw[line width=3.5ex,draw=cyan!50,line cap=round,line join=round] (5,0) -- (2,0) --(1,1);
    \draw[line width=2.25ex,draw=magenta!50,line cap=round,line join=round] (0,2) -- (2,3);
    \draw[line width=3ex,draw=green!80,line cap=round,line join=round] (5,0) -- (5,2) -- (4,3);

    \foreach \x in {0,...,5}
    {
      \foreach \y in {0,...,4}
      {
        \node (\x\y) at (\x,\y) [circle,draw=black,fill=black,inner sep=1pt] {};
      }
    }

  \end{tikzpicture}
  \caption{A lattice polygon~$\conv(X)$ and its observers (left),
    $\mathcal{I}_{\max}(X)$ (center), and a minimum interval covering
    (right).}\label{fig:intervals}
\end{figure}

\begin{remark}
One can show that the number $\card{Y}$ of observers of~$X$ in Theorem~\ref{thm:rc-plane-alg} can be of order $\Theta(2^\gamma)$, which means that the presented algorithm is not polynomial in the input size.
However, if the points in~$V$ are encoded in unary, then the algorithm is indeed polynomial.
\end{remark}

  A question related to computing the relaxation complexity in the plane has been studied by Edelsbrunner \& Preparata~\cite{EdelsbrunnerPreparata1988}:
  Given two finite sets~$X, Y \subseteq \R^2$, they describe an algorithm
  to find a convex polygon~$Q \subseteq \R^2$ with the minimal possible number of edges such that~$X
  \subseteq Q$ and~$\interior{Q} \cap Y = \emptyset$, or to decide that no
  such polygon exists.
  That is, if we apply their algorithm to a lattice-convex set~$X$ and its
  observers~$Y$, we can find a polygon that \emph{weakly} separates~$X$ and~$Y$.
  Their algorithm runs in~$\cO(\card{X \cup Y} \cdot \log\card{X \cup Y})$ time.

The computational problem of \emph{strictly} separating two finite point sets in any dimension by a given number of hyperplanes has been the focus of Megiddo's work~\cite{megiddo1988onthe}.
He introduces the \emph{$k$-separation problem} as the decision problem on whether finite sets $X,Y \subseteq \Z^d$ can be separated by $k$ hyperplanes.
The case $k=1$ reduces to the linear separation problem and can be solved by linear programming in polynomial time.
Megiddo~\cite{megiddo1988onthe} proves the following for $k \geq 2$:
\begin{itemize}
 \item[(a)] If $d$ is arbitrary, but $k$ is fixed, then the $k$-separation problem is NP-complete.
  This even holds for $k=2$.
 \item[(b)] If $d=2$, but $k$ is arbitrary, then the $k$-separation problem is NP-complete.
 \item[(c)] If both $d$ and $k$ are fixed, then the $k$-separation problem is solvable in polynomial time.
\end{itemize}
With regard to computing the relaxation complexity, comparing Theorem~\ref{thm:rc-plane-alg} with Part~(b) shows that (at least in the plane) deciding $\rc(X) \leq k$ is a computationally easier problem than the general $k$-separation problem.
Part~(c) is applicable to computing $\rc(X)$ in polynomial time via a binary search whenever $X \subseteq \Z^d$ has finitely many observers whose cardinality is polynomially bounded as a function of~$\card{X}$.
Relevant families of lattice-convex sets~$X$ with this property have been identified in \cite[Thm.~4.4 \& Thm.~4.5]{averkovschymura2020complexity}:
First, if $X \subseteq \Z^d$ contains a representative of every residue class in $(\Z/2\Z)^d$, then $\obs(X) \subseteq 2X-X$, i.e., $\card{\obs(X)} \in \cO(\card{X}^2)$.
Second, if $\conv(X)$ contains an interior integer point, then $\card{\obs(X)} \leq c_d \cdot \card{X}$, with $c_d$ a constant only depending on the dimension~$d$.

However, since in dimensions $d \geq 3$ not every lattice-convex set has finitely many observers, the following question might have an affirmative answer:
\begin{question}
For $d \geq 3$ fixed, is it NP-hard to compute $\rc_\Q(X)$ ?
\end{question}

The interpretation of $\rc(X)$ given in Observation~\ref{obs:RCintervalcovering} as the minimal size of an interval covering of $\obs(X)$ can be used to show that, in the plane, the hiding set bound deviates from $\rc(X)$ by at most~$1$.

\begin{theorem}
\label{thm:hiding-set-bound-plane}
Let $X \subseteq \Z^2$ be full-dimensional and lattice-convex.
Then,
\[
H(X) \leq \rc(X) \leq H(X) + 1.
\]
\end{theorem}

\begin{proof}
The lower bound is the general hiding set bound from Section~\ref{sect:hiding-sets}.
For the upper bound, we construct a suitable hiding set~$H$ with the help of Observation~\ref{obs:RCintervalcovering}.

To this end, let $I_1,\dots,I_k \in \mathcal{I}_{\max}(X)$ be such that~$\obs(X) = \bigcup_{i = 1}^k I_i$ is a minimal covering of the observers of~$X$.
In particular, $\rc(X) = k$.
Moreover, we assume that the sets $I_i$, $i \in [k]$, are labeled in clockwise order.
The following procedure allows us to extract a hiding set~$H$ of size $k-1$ from the sets~$I_i$:

First of all, we define $I_i' \define I_i \setminus I_{i+1}$, for $i \in [k]$,
with the indices always understood modulo~$k$.
This implies the~$I_i'$ to be pairwise distinct.
Note that since~$k$ is minimal, we must have $I_i' \neq \emptyset$, for all $i \in [k]$.
Secondly, we let $y_k$ be the last element (with respect to the clockwise order) in~$I_k'$ and let~$y_{k-1}$ be the last element in~$I_{k-1}'$.
If $[y_k,y_{k-1}] \cap \conv(X) = \emptyset$, then $I_k' \cup \{y_{k-1}\} \in \I(X)$, since the observers of~$X$ are in convex position.
We then check whether the line segment connecting the last elements in~$I_k' \cup \{y_{k-1}\}$ and $I_{k-1}' \setminus \{y_{k-1}\} \in \I(X)$ has a point in common with~$\conv(X)$.
We repeat this procedure until we arrive at sets~$I_k''$ and~$I_{k-1}''$ whose respective last elements are endpoints of a line segment that intersects $\conv(X)$ non-trivially.
This has to happen at some point, as otherwise $I_k' \cup I_{k-1}' \in \I(X)$, yielding a covering of $\obs(X)$ with $k-1$ sets of~$\I(X)$.

Now, we iteratively proceed as above and obtain sets $I_k'',I_{k-1}'',\ldots,I_1'' \in \I(X)$ such that $\obs(X) = \bigcup_{i=1}^k I_i''$ and $[y_i,y_{i-1}] \cap \conv(X) \neq \emptyset$, for $2 \leq i \leq k$, and where $y_i$ is the respective last element in~$I_i''$.
With this notation we claim that $H= \{y_k,y_{k-1},\ldots,y_2\}$ is a hiding set for~$X$.

If $k=3$, then there is nothing to show as by construction $[y_3,y_2] \cap \conv(X) \neq \emptyset$.
So, let $k \geq 4$.
Assume for contradiction that $[y_s,y_t] \cap \conv(X) = \emptyset$, for some $2 \leq s < t \leq k$.
By construction $t > s+1$, and since $\obs(X)$ is in convex position, this means that $I_s'' \cup I_{s+1}'' \in \I(X)$ or $I_s'' \cup I_{s-1}'' \in \I(X)$, yielding a covering of $\obs(X)$ with only $k-1$ sets of~$\I(X)$.
\end{proof}

Already the square $\square_2=\{0,1\}^2$ shows that the upper bound in Theorem~\ref{thm:hiding-set-bound-plane} is best possible.
Indeed, we have $\rc(\square_2) = 3$, and a routine check reveals that every hiding set for $\square_2$ has at most two elements.
We do not know whether a similar result holds in higher dimensions:
The discrete simplex~$\Delta_d$ does not admit hiding sets of size larger than~$3$ (see~\cite[Prop.~8.2.4]{weltge2015diss}), but $\rc(\Delta_d)$ grows at least logarithmically with~$d$ (see~\cite[Prop.~8.1.4]{weltge2015diss} and~\cite[Thm.~1.2]{averkovschymura2020complexity}).
In fact, Weltge conjectures that $\rc(\Delta_d)=d+1$.
However, this example does not exclude an affirmative answer for the following question.

\begin{question}
Is there a constant $h_d$ only depending on the dimension~$d$ such that for every full-dimensional lattice-convex set $X \subseteq \Z^d$ we have
\[
\rc(X) \leq H(X) + h_d \,? 
\]
\end{question}

\section{Relaxation complexity of special families of lattice-convex sets}
\label{sec:box}

The exact value of~$\rc(X)$, or any of its variants, is known only for very few classes of lattice-convex sets~$X \subseteq \Z^d$:
The discrete unit cube always admits a simplex relaxation, that is, \mbox{$\rc(\{0,1\}^d) = d+1$} (see Weltge~\cite[Thm.~8.1.3]{weltge2015diss}).
The conjectured value $\rc(\Delta_d) = d+1$ for the discrete standard simplex $\Delta_d = \{0,e_1,\ldots,e_d\}$ could only be affirmed for dimensions $d\leq 4$ so far (see~\cite[Cor.~3.8]{averkovschymura2020complexity}).
Jeroslow~\cite[Thm.~7]{jeroslow1975ondefining} showed that, for every $1 \leq k \leq 2^{d-1}$ there is a $(2^d-k)$-element subset~$X_k$ of $\{0,1\}^d$ such that $\rc(X_k,\{0,1\}^d) = k$.
Further, for $k=2^{d-1}$ one can choose $X_{2^{d-1}} = X_{\textrm{even}} \define \left\{ x \in \{0,1\}^d : \sum_{i=1}^d x_i\textrm{ is even}\right\}$, however the value $\rc(X_{\textrm{even}})$ is not known.
Besides some very specific examples that were needed to establish computability of $\rc(X)$ for $3$-dimensional lattice-convex sets $X \subseteq \Z^3$ (see~\cite[Sect.~6]{averkovschymura2020complexity}), we are not aware of any further classes for which $\rc(X)$ is known exactly.

In this section, we determine the relaxation complexity of two classes of lattice-convex sets and thus add to the short list of exact results given above.

\subsection{Rectangular boxes}

Weltge's (rational) simplex relaxation of the discrete unit cube $\{0,1\}^d$ can be generalized to the case that one segment of the cube is allowed to have arbitrary length.
To this end, we call a set $\{a,\ldots,b\}$ with $a, b \in \Z$ and $a \le b$ a \emph{discrete segment} of length $b-a$, and we call the Cartesian product of finitely many discrete segments a \emph{discrete rectangular box}.
These examples contain a representative of every residue class in $(\Z/2\Z)^d$, so that the relaxation complexity agrees with the rational relaxation complexity in view of~\cite[Thm.~1.4]{averkovschymura2020complexity}.

\begin{lemma}
\label{lem:rc-box-one-long-edge}
For $b \in \Z_{>0}$, let $X_b \define \{0,1\}^\ell \times \{0,1,\ldots,b\} \subseteq \Z^{\ell+1}$.
Then, the simplex
\begin{align*}
P &\define \left\{ x \in \R^{\ell+1} : x_k \leq 1 + \sum_{i=k+1}^{\ell+1} (b+1)^{-i} x_i, \textrm{ for }k \in [\ell], \right.\\
&\phantom{= \bigg\{ x \in \R^{\ell+1} :\ }\left. x_{\ell+1} \leq b \quad \textrm{and} \quad x_1 + \sum_{i=2}^{\ell+1}(b+1)^{-i} x_i \geq 0 \right\}
\end{align*}
is a relaxation of~$X_b$.
In particular, $\rc(X_b) \leq \ell+2$.
\end{lemma}

\begin{proof}
The proof is just an adjustment of Weltge's arguments in~\cite[Lem.~7.2.1]{weltge2015diss}.
We need to show that $X_b = P \cap \Z^{\ell+1}$.
The inclusion $X_b \subseteq P \cap \Z^{\ell+1}$ is quickly checked, so we give the details for the reverse inclusion.

Let $x \in P \cap \Z^{\ell+1}$.
First, we show that $x_i \leq 1$, for all $i \in [\ell]$.
To this end, let $k \in [\ell]$ be the largest index for which $x_k > 1$.
Then, the first defining inequalities of~$P$ together with basic facts about geometric series give
\begin{align*}
1 &< x_k \leq 1 + \sum_{i=k+1}^{\ell+1}(b+1)^{-i} x_i \leq \sum_{i=k+1}^\ell (b+1)^{-i} + \frac{b}{(b+1)^{\ell+1}}\\
&< 1+\frac{1}{b} + \frac{b-1}{(b+1)^{\ell+1}} < 1 + \frac{1}{b} + \frac{b-1}{b} = 2,
\end{align*}
in contradiction to~$x_k \in \Z$.

Second, we show that $x_1$ is non-negative.
Indeed, by the last defining inequality of~$P$, the just established fact
that $x_i \leq 1$, for $i \in [\ell]$, and $x_{\ell+1} \leq b$, we have
\[
x_1 \geq -\sum_{i=2}^{\ell+1}(b+1)^{-i} x_i \geq -\sum_{i=2}^\ell(b+1)^{-i} - \frac{b}{(b+1)^{\ell+1}} > -1.
\]
It remains to show that $x_i \geq 0$, for every $i \in \{2,\dots,\ell+1\}$.
To this end, let $j \in \{2,\ldots,\ell+1\}$ be the smallest index such that $x_j \leq -1$.
We first show that $x_i = 0$, for every $i < j$:
Let $k < j$ be the largest index with $x_k > 0$, which means that $x_k = 1$.
Then, by the first defining inequalities of~$P$, we have
\begin{align*}
1 &= x_k \leq 1 + \sum_{i=k+1}^{\ell+1}(b+1)^{-i} x_i = 1 + (b+1)^{-j} x_j + \sum_{i=j+1}^{\ell+1}(b+1)^{-i} x_i\\
&\leq 1 - (b+1)^{-j} + \sum_{i=j+1}^\ell (b+1)^{-i} + \frac{b}{(b+1)^{\ell+1}} < 1.
\end{align*}
The last inequality follows from the closed form expression for geometric sums and basic algebraic manipulations.
Now, knowing that $x_i = 0$, for every $i < j$, we use the last defining inequality of~$P$ and get
\begin{align*}
0 &\leq x_1 + \sum_{i=2}^{\ell+1}(b+1)^{-i} x_i = (b+1)^{-j} x_j + \sum_{i=j+1}^{\ell+1}(b+1)^{-i} x_i\\
&\leq -(b+1)^{-j} + \sum_{i=j+1}^\ell(b+1)^{-i} + \frac{b}{(b+1)^{\ell+1}} < 0.
\end{align*}
This contradiction finishes the proof.
\end{proof}

This simplicial relaxation of~$X_b$ quickly leads to an upper bound on the relaxation complexity of general discrete rectangular boxes.
To see that these upper bounds are tight, we employ the concept of hiding sets discussed in Section~\ref{sect:hiding-sets}.

\begin{theorem}
\label{thm:box-theorem}
	For integers $k>0$ and $\ell \ge 0$,
let $S_1, \ldots, S_k$ be discrete segments of length at least $2$, and let $T_1,\ldots,T_\ell$ be discrete segments of length $1$. Consider the discrete box
$X = S_1 \times \ldots \times S_k \times T_1 \times \ldots \times T_\ell$ in $\Z^{k + \ell}$.
Then, $\rc(X) = 2 k + \ell$.
\end{theorem} 
\begin{proof}
	Without loss of generality let $S_i = \{a_i,\ldots,b_i\}$ be such that $a_i < 0$ and $b_i > 0$, and let $T_j = \{0,1\}$, for all $1 \leq j \leq \ell$.
	We form a $2k + \ell$ element set~$H$ by attaching to each $S_i$ two points $(a_i-1) e_i$ and $(b_i+1) e_i$ in $\Z^{k + \ell}$, and to each~$T_j$  the point $2 e_{k+j}$ in $\Z^{k + \ell}$. It is straightforward to check that $H$ is a hiding set of~$X$. Indeed, by construction the midpoint of any two $p,q \in H$ with $p \ne q$ belongs to $\conv(X)$. This shows $\rc(X) \ge 2 k + \ell$. 

	To prove the upper bound $\rc(X) \le 2 k + \ell$ it suffices to verify the case $k=1$.
Indeed, if $k>0$, a relaxation $Q'$ of $S_k \times T_1 \times \ldots \times T_\ell$ with $2 + \ell$ facets gives rise to the relaxation $\conv(S_1 \times \ldots \times S_{k-1}) \times Q'$  of $X$ with $2 k + \ell$ facets. 
The case $k=1$ is however exactly the content of Lemma~\ref{lem:rc-box-one-long-edge}.
\end{proof}

\subsection{Discrete crosspolytopes}

Let~$\lozenge_d = \{ \zero, \pm e_1, \dots, \pm e_d\}$ be the discrete standard crosspolytope.
Weltge~\cite[Prop.~7.2.2]{weltge2015diss} proved that $\rc(\lozenge_d) \leq 2d$, for all $d \geq 4$ and wondered whether this bound is best possible.
The $2d$ bound is already quite surprising, as the crosspolytope~$\conv(\lozenge_d)$ has $2^d$ facets.
Using our efficient implementation for $\rc(X)$ that we describe in Section~\ref{sec:MIP} below, we obtained \emph{simplex} relaxations of~$\lozenge_d$ in small dimensions.
In the following we prove that $\rc(\lozenge_d)=d+1$ holds for every $d\geq3$.
To this end, we introduce, for two sets~$X,Y$, the notation $X \oplus Y
\define (X \times \{ \zero \}) \cup (\{\zero\} \times Y)$.
Note that then~$\lozenge_{d+1} = \lozenge_d \oplus \lozenge_1$.

\begin{lemma}
  \label{lem:crossIncrease2}
  Let~$d$ be a positive integer and let~$P = \{x \in \R^d \st Ax \leq
  \one\}$ be a relaxation of~$\lozenge_d$.
  Then,
  \[
    Q = \left\{ (x,y) \in \R^{d+1} : A(x - y e_1) \leq \one,\; -1 \leq x_1 + y \leq 1 \right\}
  \]
  is a relaxation of $\lozenge_{d+1} = \lozenge_d \oplus \lozenge_1$.
\end{lemma}

\begin{proof}
  To show~$\lozenge_{d+1} \subseteq Q$, let~$(x, y) \in \lozenge_{d+1}$.
  Then, either~$x \in \lozenge_d$ and~$y = 0$, or~$x = \zero$ and~$y = \pm 1$.
  In both cases~$x - ye^1 \in \lozenge_d$ and~$\card{x_1 + y} \leq 1$, thus
  $(x,y) \in Q$.

  For the reverse inclusion assume $(x,y) \in \Z^{d+1}$ is a lattice point in~$Q$.
  Then, $x - y e_1 \in P \cap \Z^d = \lozenge_d$ and $|x_1 + y| \leq 1$.
  We proceed by a case distinction on~$y$.
  First, if $y = 0$, then $x = x - y e_1 \in \lozenge_d$ and hence $(x,y) = (x,0)
  \in \lozenge_d \times \{0\} \subseteq \lozenge_{d+1}$.
  Second, if $y = 1$, then by $x_1 + 1 = x_1 + y \leq 1$ we find that $x_1
  \leq 0$.
  Since also~$x - e_1 = x - y e_1 \in \lozenge_d$, we conclude~$x_1 = 0$ and
  $(-1,x_2,\ldots,x_d) \in \lozenge_d$.
  Hence, $x_2 = \ldots = x_d = 0$, showing $(x,y) = (\zero,1) = e_{d+1} \in
  \lozenge_{d+1}$.
  Third, suppose $y \geq 2$.
  Note that~$A(x - ye_1) \leq \one$ implies~$\card{x_1 - y} \leq 1$ as~$P$
  is a relaxation of~$\lozenge_d$.
  Together with~$\card{x_1 + y} \leq 1$, we find~$\card{x_1} \leq 1$.
  Further, from~$x_1 + y \leq 1$, we get $y \leq 1 - x_1 \leq 2$.
  So, in fact $y=2$ and $x_1 = -1$.
  But then we have $x - y e_1 = (-3,x_2,\ldots,x_d) \in \lozenge_d$, a contradiction.
  Finally, the cases of negative values of $y$ are analogous.  
\end{proof}

\goodbreak
\begin{lemma}
  \label{lem:crossIncrease1}  
  For any positive integer~$d$, we have $\rc(\lozenge_{d+2}) \leq \rc(\lozenge_{d}) + 2$.
\end{lemma}

\begin{proof}
Let $P = \{x \in \R^d : Ax \leq \one\}$ be a bounded relaxation of~$\lozenge_d$ with
$\rc(\lozenge_d)$ facets, which exists by~\cite[Thm.~1.4]{averkovschymura2020complexity}.
By Lemma~\ref{lem:crossIncrease2}, a relaxation of~$\lozenge_{d+1}$ is given by
\[Q = \{(x,y) \in \R^{d+1} : A(x - y e_1) \leq \one, |x_1 + y| \leq 1\}.\]
Observe that $Q$ is bounded, since~$P$ is.
Hence, there is an $\varepsilon > 1$ small enough such that with $A' =
\frac{1}{\varepsilon} A$, the polyhedron $Q' =
\{(x,y) \in \R^{d+1} : A'(x - y e_1) \leq \one, |x_1 + y| \leq
\varepsilon\}$ is a \emph{strict} relaxation of $\lozenge_{d+1}$.
Since~$Q'$ is strict, there are $c_1,\ldots,c_d \in \R \setminus \Q$
close enough to~$0$ such that also
\[
  Q'' = \{(x,y) \in \R^{d+1} :
  A'((1+c_1)x - y e_1) \leq \one, |(1+c_1)x_1 + c_2 x_2 + \ldots c_d x_d
  + y| \leq \varepsilon\}
\]
is a relaxation of $\lozenge_{d+1}$.
We further assume that $1,c_1,\ldots,c_d$ are chosen to be linearly
independent over~$\Q$.

We claim that with this set-up the polyhedron
\begin{align*}
R \define \big\{(x,y,z) \in \R^{d+2} : &\ A'((1+c_1)x - y e_1) \leq (1-z)\one,\\
&\ |(1+c_1)x_1 + c_2 x_2 + \ldots + c_d x_d + y| \leq (1+z)\varepsilon\big\}
\end{align*}
is a relaxation of $\lozenge_{d+2}$, which by construction has at most
$\rc(\lozenge_d) + 2$ facets.
First of all, for $\alpha \in \R$, consider the system $Ax \leq \alpha \one$.
Since for $\alpha=1$ the system describes the polyhedron~$P$ which contains
$\zero$ in its interior, the system with $\alpha=0$ describes only the
origin $\zero$, and for any $\alpha < 0$ the corresponding system is
infeasible.

Now, the containment $\lozenge_{d+2} \subseteq R \cap \Z^{d+2}$ is quickly
checked, since $Q''$ contains~$\lozenge_{d+1}$.
For the reverse inclusion, let $(x,y,z) \in R \cap \Z^{d+2}$ be an integer
point.
By the observation on the system $Ax \leq \alpha \one$ and the nature of
the inequalities $|(1+c_1)x_1 + y| \leq (1+z)\varepsilon$, the last
coordinate necessarily satisfies $z \in \{0,\pm 1\}$.

If $z=0$, then $(x,y,0) \in \lozenge_{d+2}$, because $(x,y) \in Q'' \cap \Z^{d+1}
= \lozenge_{d+1}$.
If $z=1$, then by the observation on the system $Ax \leq \alpha \one$, we
must have $(1+c_1)x = y e_1$, and thus $x = \zero$ and $y=0$, as $c_1$
was chosen irrational.
Hence, $(x,y,z) = e_{d+2} \in \lozenge_{d+2}$.
Finally, if $z = -1$, then $|(1+c_1)x_1 + c_2 x_2 + \ldots c_d x_d + y|
\leq 0$, which also implies $x = \zero$ and $y=0$, since $1,c_1,\ldots,c_d$
are chosen to be linearly independent over~$\Q$.
So, $(x,y,z) = -e_{d+2} \in \lozenge_{d+2}$ in this case, finishing the proof.
\end{proof}

\begin{theorem}
\label{thm:simplex-relaxation-crosspolytope}
  For every positive integer~$d \neq 2$, we have $\rc(\lozenge_d) = d + 1$.
\end{theorem}

\begin{proof}
  Obviously, the result holds for~$d=1$.
  For~$d \geq 3$, observe that
  every relaxation of $\lozenge_d$ is bounded,
  since $\conv(\lozenge_d)$ contains the origin in its interior, cf.~\cite[Thm.~4.5]{averkovschymura2020complexity}.
This implies $\rc(\lozenge_d) = \rc_\Q(\lozenge_d) \geq d+1$
(see~\cite[Rem.~3.7]{averkovschymura2020complexity}).
For the upper bound, it is sufficient to prove the statement for~$d \in
\{3,4\}$, as the remaining cases follow by induction via
Lemma~\ref{lem:crossIncrease1}.
One can verify that such relaxations are, e.g.,
\begin{align*}
  &\left\{ x \in \R^3 : \left(\begin{smallmatrix*}[r] 8 & 12 & -13 \\ -8 & -12 & -13 \\ 12 & -8 & 13\\ -12 & 8 & 13 \end{smallmatrix*}\right) \left(\begin{smallmatrix*}[r] x_1\\ x_2\\ x_3\end{smallmatrix*}\right) \leq \left(\begin{smallmatrix*}[r] 13\\ 13\\ 13\\ 13\end{smallmatrix*}\right)\right\}\qquad \text{and}\\
&
\left\{ x \in \R^4 : \left(\begin{smallmatrix*}[r] 6 & -7 & -7 & -5 \\ -7 & 3 & -2 & -7 \\ 9 & 7 & 9 & -4 \\ 1 & 2 & -2 & 2 \\ -6 & -6 & 5 & 2 \end{smallmatrix*}\right) \left(\begin{smallmatrix*}[r] x_1\\ x_2\\ x_3 \\ x_4\end{smallmatrix*}\right) \leq \left(\begin{smallmatrix*}[r] 7\\ 7\\ 9\\ 2\\ 6\end{smallmatrix*}\right)\right\}
\end{align*}
for~$\lozenge_3$ and~$\lozenge_4$, respectively.
\end{proof}

\section{Mixed-integer programming formulations for computing the relaxation complexity}
\label{sec:MIP}

In this section, we briefly discuss mixed-integer programming models for the
computation of~$\rc_\varepsilon(X,Y)$ and~$\rc(X,Y)$, respectively, and
evaluate their performance.

\paragraph{Compact model}
In~\cite{averkovschymura2020complexity}, a mixed-integer programming
formulation has been proposed to check whether a lattice-convex
set~$X \subseteq \Z^d$ admits a relaxation w.r.t.~a finite set $Y \subseteq \Z^d \setminus X$ with~$k$ inequalities.
With a minor modification, this model can also be used to
compute~$\rc_\varepsilon(X,Y)$, and thus~$\rc(X,Y)$, if~$\varepsilon$ is
sufficiently small.
In the following, we just present the idea of this model and refer the
reader to~\cite{averkovschymura2020complexity} for more details.

Given an upper bound~$k$ on the number of inequalities needed to separate~$X$
and~$Y$, the idea is to introduce variables~$a_{ij}$
and~$b_i$, $(i,j) \in [k] \times [d]$, to model the~$k$ potential
inequalities $\sum_{j = 1}^d a_{ij} x_j \leq b_i$ needed in a relaxation.
By rescaling, one can assume without loss of generality that $a_{ij} \in [-1,1]$ and~$b_i \in [-d \rho_X, d
\rho_X]$ with~$\rho_X =\max \{ \| x\|_\infty \st x \in X\}$.
Moreover, for each~$i \in [k]$ and~$y \in Y$, a binary variable~$s_{iy}$ is
introduced that indicates whether the~$i$-th inequality is violated by~$y$;
binary variables~$u_i$, $i \in [k]$, indicate whether the~$i$-th inequality
is used in a relaxation.
To model that inequality~$i$ is violated by at least~$\varepsilon$ for~$y
\in Y$ if~$s_{iy} = 1$, a big-M constraint is used.

In its basic version, this model is rather difficult to
solve for a black box MIP solver.
Reasons for this are (\romannumeral1) the big-M
inequalities, (\romannumeral2) symmetries of the
problem formulation, and (\romannumeral3) combinatorial properties of the
relaxation complexity that are not expressed in the model.
To overcome these issues, we have used the following enhancements in our
implementation.
We replace the big-M constraints by so-called indicator constraints which
encode the big-M constraint without introducing big-M terms explicitly,
cf.\ Belotti et al.~\cite{BelottiEtAl2016}.
To handle symmetries, we enforce the used inequalities to be sorted
w.r.t.\ their first coefficient in~$a$;
the unused inequalities are fixed to~$\sprod{0}{x} \leq d \rho_X$.
Finally,  we derived additional inequalities that can be used as cutting planes.
These inequalities are based on hiding sets~$H \subseteq Y$ and encode that
each inequality valid for~$X$ can cut off at most one point from~$H$.
This can be expressed via the \emph{hiding set cuts}
$\sum_{y \in H} s_{iy} \leq 1$, $i \in [k]$.
Although these cuts are the stronger the bigger the underlying hiding set is, we
add these inequalities just for hiding sets of size~2, because the latter
can be computed relatively efficiently by a brute-force algorithm.

\paragraph{Column generation based model}
Based on Observation~\ref{obs:RCintervalcovering}, we can also devise an alternative
integer programming formulation to compute~$\rc(X,Y)$.
To this end, let
\[
  \I = \I(X, Y) \define
  \{ I \subseteq Y \st \conv(I) \cap \conv(X) = \emptyset\}.
\]
Then, $\rc(X,Y)$ is the smallest number~$k$ of sets~$I_1, \dots, I_k \in
\I$ with~$Y = \bigcup_{i = 1}^k I_i$, which leads immediately to the
following integer programming formulation
\begin{equation}
  \label{eq:CGmodel}
  \min_{z \in \Z_+^{\I}} \Big\{
  \sum_{I \in \I} z_I \st \sum_{I \in I_y} z_I \geq 1,\; y \in Y
  \Big\},
\end{equation}
where~$I_y = \{ I \in \I \st y \in I\}$.

Since this formulation consists of exponentially many variables, we cannot
solve it immediately using standard MIP solvers.
In our implementation, we thus use a branch-and-price procedure for
solving~\eqref{eq:CGmodel}, i.e., we use a branch-and-bound procedure in
which each LP relaxation is solved by column generation.

The core of the column generation procedure is to solve the pricing
problem.
Since the dual of the LP relaxation is
\[
  \max_{\alpha \in \R_+^Y} \Big\{ \sum_{y \in Y} \alpha_y \st \sum_{y \in I} \alpha_y \leq 1,\;
  I \in \mathcal{I}\Big\},
\]
the pricing problem is to find, for fixed shadow prices~$\bar{\alpha} \in
\R_+^Y$, a set~$I \in \mathcal{I}$ with~$\sum_{y \in I} \bar{\alpha}_y > 1$
or to show that no such set exists.
Since this problem is NP-hard as it is a generalization of the open
hemisphere problem, see Johnson \& Preparata~\cite{JohnsonPreparata1978},
we model the pricing problem as a MIP in our implementation.
The MIP model that we have used is a variant of the compact model discussed
before with~$k = 1$.

The branching strategy of our branch-and-price algorithm is to select two
sets~$I$ and~$J$ whose corresponding~$\alpha$-variables are fractional in
the LP relaxation such that both the intersection~$I \cap J$ and the
symmetric difference~$I \Delta J$ are non-empty.
Then, we select~$y_1 \in I \cap J$ and~$y_2 \in I \Delta J$ and create two
child nodes.
In one, we enforce that~$y_1$ and~$y_2$ are contained in the same sets~$I \in
\mathcal{I}$; in the other, $y_1$ and~$y_2$ have to be contained in
different sets from~$\mathcal{I}$ used by the LP relaxation.

\paragraph{Hybrid model}

Our implementation also allows to use a hybrid model.
Instead of solving the full problem using Model~\eqref{eq:CGmodel}, we only
solve its LP relaxation using column generation.
Then, the compact model is solved after it has been initialized with the best
primal solution found during solving the LP relaxation
of~\eqref{eq:CGmodel} and the value of the LP relaxation has been transferred as
dual bound to the compact model.

\paragraph{An ad-hoc method to compute the relaxation complexity}
If~$X$ has a finite set of observers, we can at least theoretically use
the algorithm from~\cite{averkovschymura2020complexity} to compute~$\obs(X)$ and
solve the compact or column generation model to
find~$\rc_\varepsilon(X, \obs(X)) = \rc_\varepsilon(X)$ and~$\rc(X,
\obs(X)) = \rc(X)$, respectively.
In practice, however, $\obs(X)$ might be very large or even
infinite already in small dimensions and solving these models becomes
costly or even impossible.
A more ad-hoc fashion for finding~$\rc_\varepsilon(X)$ and~$\rc(X)$ is
inspired by Theorem~\ref{thm:epsrelaxProperties}.
Instead of starting with the full set of observers, we select a small
set~$Y \subseteq \obs(X)$ and compute~$\rc_\varepsilon(X,Y)$
and~$\rc(X,Y)$, respectively.
If the resulting set of inequalities is already a relaxation of~$X$, we
stop.
Otherwise, there exists a non-empty set~$Y' \subseteq \obs(X) \setminus Y$
that is not separated from~$X$.
We can extend the set~$Y$ by~$Y'$ and iterate this procedure until a
relaxation of~$X$ has been found.

Of course, this procedure is only guaranteed to work if~$\rcl(X) =
\rcu(X)$.
Nevertheless, we could use it to compute~$\rc(X)$ for particular choices
of~$X$ without computing~$\obs(X)$ explicitly.

\paragraph{Numerical experiments}
Our implementation consists of two separate codes, which are publicly
available\footnote{\url{https://github.com/christopherhojny/relaxation_complexity},
  experiments based on version with git hash \texttt{55ed5d10}}.
The first code essentially consists of one method for
computing~$\rc(X, \obs(X))$ (provided $\obs(X)$ is finite) and another
method to compute~$H(X)$, which has been used as inspiration for
Theorem~\ref{thm:hiding-set-bound-plane}.
To find~$\obs(X)$, we use a variant of the algorithm suggested
in~\cite{averkovschymura2020complexity}, and~$\rc(X, \obs(X))$ is computed
using the compact MIP model.
We compute~$H(X)$ by finding a maximum cardinality clique in the hiding
graph by solving an integer program.
We have implemented the aforementioned methods in Python~3.7.8, using
SageMath~9.1~\cite{sagemath} for polyhedral computations;
all mixed-integer programs were solved using
SCIP~7.0.0~\cite{GamrathEtal2020OO}, which has
been called via its Python
interface~\cite{MaherMiltenbergerPedrosoRehfeldtSchwarzSerrano2016}.
Note that SCIP is not an exact solver and thus the results reported below
are only correct up to numerical tolerances.

We have used this implementation to investigate two properties of~$\rc(X)$
for lattice polygons in dimension~2.
More specifically, we used the representatives of all lattice
polygons~$P$ with at least one and at most~12 interior integer points provided
by Castryck~\cite{castryck2012movingout}.
In our first experiment, we compared the exact value of the relaxation
complexity with the hiding set lower bound.
Table~\ref{tab:deviationHiding} shows the deviation of these two values
parameterized by the number of edges of~$P$,
which led us to Theorem~\ref{thm:hiding-set-bound-plane}, i.e.,
the hiding set bound deviates by at most one.
Interestingly, for~$k \in \{4,6,7,8,9\}$, the hiding set bound is distributed relatively
equally between~$\rc(X)$ and~$\rc(X)-1$. 

\begin{table}[t]
  \begin{scriptsize}
    \caption{Distribution of deviation of maximum hiding set sizes from relaxation complexity in percent.}
    \label{tab:deviationHiding}
    \begin{tabular*}{\textwidth}{@{}l@{\;\;\extracolsep{\fill}}rrrrrrrr@{}}\toprule
    & \multicolumn{8}{c}{number of facets}\\
    \cmidrule{2-9}
    size max.\ hiding set & 3 & 4 & 5 & 6 & 7 & 8 & 9 & 10\\
    \midrule
    rc & 100.00 & 45.19 & 36.70 & 45.79 & 55.09 & 52.71 & 45.90 & 100.00 \\
    rc - 1 & --- & 54.81 & 63.30 & 54.21 & 44.91 & 47.29 & 54.10 & --- \\
    \bottomrule
    \end{tabular*}
  \end{scriptsize}
\end{table}

The second experiment compares the relaxation complexity with the number of
edges of~$P$, see Table~\ref{tab:distributionRC}.
Although~$P$ can be rather complex with up to~10 edges (where~10 is
realized by a single instance), the relaxation complexity is at most six,
where for the majority of all tested instances the relaxation complexity is
either four or five.
Moreover, there exist polygons with up to nine edges that admit
simplicial relaxations.
Thus, already in dimension~2, the difference between the number of facets
and the relaxation complexity can be very large.
In high dimensions this phenomenon does not come as a surprise, since there are knapsack polytopes whose integer hull has super-polynomially many facets (cf.~Pokutta \& Van Vyve~\cite[Cor.~3.8]{pokuttavyve2013anote}).

\begin{table}[t]
  \begin{scriptsize}
    \caption{Distribution of relaxation complexities compared to number of facets in percent.}
    \label{tab:distributionRC}
    \begin{tabular*}{\textwidth}{@{}l@{\;\;\extracolsep{\fill}}rrrrrrrr@{}}\toprule
    & \multicolumn{8}{c}{number of facets}\\
    \cmidrule{2-9}
    relaxation complexity & 3 & 4 & 5 & 6 & 7 & 8 & 9 & 10\\
    \midrule
    3 & 100.00 & 5.20 & 3.89 & 3.40 & 3.71 & 2.61 & 1.64 & --- \\
    4 & --- & 94.80 & 66.92 & 54.79 & 51.32 & 42.17 & 40.98 & 100.00 \\
    5 & --- & --- & 29.19 & 40.39 & 43.49 & 54.16 & 55.74 & --- \\
    6 & --- & --- & --- & 1.42 & 1.48 & 1.06 & 1.64 & --- \\
    \bottomrule
    \end{tabular*}
  \end{scriptsize}
\end{table}

The second code is a C/C++ implementation of the compact, column
generation, and the hybrid model to find~$\rc(X, Y)$.
The implementation uses SCIP~7.0.2 as branch-and-bound (-and-price)
framework and SoPlex~5.0.2 as LP solver; we use cddlib version
0.94m\footnote{source code available at
  \url{https://github.com/cddlib/cddlib}} for convex hull computations, e.g., to find hiding set cuts.
In the following, we briefly illustrate that selecting the right model is
crucial to efficiently find~$\rc(X, Y)$.
As test set, we use the sets~$X \in \{\Delta_3, \Delta_4, \lozenge_4, \lozenge_5\}$
and~$Y$ being all integer points at~$\ell_1$-distance at most~$k$ from~$X$,
where~$k \in [9]$ for simplices and~$k \in [7]$ for crosspolytopes.
These are the test sets that lead us to
Theorems~\ref{thm:rc-square-simplex}
and~\ref{thm:simplex-relaxation-crosspolytope}.
These experiments were run on a Linux cluster with Intel Xeon E5
\SI{3.5}{\GHz} quad core processors and \SI{32}{\giga\byte} memory.
The code was executed using a single thread.  The time limit of all
computations is \SI{4}{\hour} per instance.

\begin{table}[t]
  \caption{Comparison of the different methods for computing $\rc(X, Y)$.}
  \label{tab:compare}
  \begin{tabular*}{\textwidth}{@{}l@{\;\;\extracolsep{\fill}}rrrrrr@{}}\toprule
     & \multicolumn{2}{c}{compact} & \multicolumn{2}{c}{column generation} & \multicolumn{2}{c}{hybrid}\\
    \cmidrule{2-3} \cmidrule{4-5} \cmidrule{6-7}
    test set & time & \#opt & time & \#opt & time & \#opt\\
    \midrule
    $\Delta_3$ ($\ell_1$: 1--9) &  \num{ 4.16} & \num{ 9}  &  \num{ 64.97} & \num{ 9}  &  \num{ 2.37} & \num{ 9} \\
    $\Delta_4$ ($\ell_1$: 1--9) &  \num{ 743.74} & \num{ 7}  &  \num{2355.48} & \num{ 4}  &  \num{ 330.60} & \num{ 7} \\
    $\lozenge_4$ ($\ell_1$: 1--7) &  \num{ 663.04} & \num{ 5}  &  \num{1528.20} & \num{ 4}  &  \num{ 47.85} & \num{ 7} \\
    $\lozenge_5$ ($\ell_1$: 1--7) &  \num{3481.59} & \num{ 4}  &  \num{11598.31} & \num{ 1}  &  \num{ 282.58} & \num{ 6} \\
    \bottomrule
  \end{tabular*}
\end{table}

Table~\ref{tab:compare} summarizes our experiments and reports running time in shifted
geometric mean~$\prod_{i = 1}^n (t_i + 10)^{\nicefrac{1}{n}} - 10$ to
reduce the impact of outliers.
In general, we can see that the compact model performs better than the
column generation approach both in terms of running time and number of
solved instances (marked as \#opt).
An explanation for this is the very costly pricing problem.
However, we could observe that both the primal and dual bound after solving
the root node are in general much tighter than in the compact model.
A possible explanation for the better upper bound is that the variables
generated by column generation could be well-structured such that SCIP can
easily find good integer solutions for the covering formulation of~\eqref{eq:CGmodel}.
The meaning of the variables of the compact model, however, seems to be
less clear to SCIP such that heuristics cannot find feasible solutions
easily.
This observation also explains why the hybrid model performs much better
than the two separate methods: it benefits from the good bounds obtained by
column generation and the quick solvability of each node of the
branch-and-bound tree in the compact model.
In particular for~$\lozenge_5$, the hybrid model allows to solve two more
instances than the compact model and reduces its running time
by~\SI{91.9}{\percent}.

\bigskip
\textbf{Acknowledgements}
We thank Matija Buci\'{c} and Benny Sudakov for discussions related to
Theorem~\ref{thm:rc:squar:delta:d} and Stefan Weltge for motivating
discussions on the topic of this paper.

\bibliographystyle{spmpsci}      

\end{document}